\documentclass[a4paper]{amsart}
\usepackage{hyperref}
\usepackage{txfonts, amsmath,amstext,amsthm,amscd,amsopn,verbatim,amssymb, amsfonts}
\usepackage{fullpage}

\usepackage[bbgreekl]{mathbbol}
\usepackage{tikz}
\usetikzlibrary{matrix}
\usetikzlibrary{shapes}
\usetikzlibrary{arrows}
\usetikzlibrary{calc,3d}
\usetikzlibrary{decorations,decorations.pathmorphing,decorations.pathreplacing,decorations.markings,snakes}
\usetikzlibrary{through}

\usepackage{tikz-cd}

\tikzset{snakeit/.style={decorate, decoration={snake, amplitude=.2mm,segment length=1mm}}}

\tikzset{ext/.style={circle, draw,inner sep=1pt},int/.style={circle,draw,fill,inner sep=1.4pt},nil/.style={inner sep=1pt}}
\tikzset{cy/.style={circle,draw,fill,inner sep=2pt},scy/.style={circle,draw,inner sep=2pt},scyx/.style={draw,cross out,inner sep=2pt},scyt/.style={draw,regular polygon,regular polygon sides=3,inner sep=0.95pt}}
\tikzset{exte/.style={circle, draw,inner sep=3pt},inte/.style={circle,draw,fill,inner sep=3pt}}
\tikzset{diagram/.style={matrix of math nodes, row sep=3em, column sep=2.5em, text height=1.5ex, text depth=0.25ex}}
\tikzset{diagram2/.style={matrix of math nodes, row sep=0.5em, column sep=0.5em, text height=1.5ex, text depth=0.25ex}}

\newcommand{\snak}{
{\,
\begin{tikzpicture}[baseline=-.65ex,scale=.2]
 \draw (0,-.5) edge[snakeit] (0,.5);
\end{tikzpicture}
}
}
\newcommand{\lin}{
{\,
\begin{tikzpicture}[baseline=-.65ex,scale=.2]
 \draw (0,-.5) edge (0,.5);
\end{tikzpicture}
}
}
\newcommand{\dotte}{
{\,
\begin{tikzpicture}[baseline=-.65ex,scale=.2]
 \draw (0,-.5) edge[dotted] (0,.5);
\end{tikzpicture}
}
}

\newcommand{\nomul}{{
\begin{tikzpicture}[baseline=-.55ex,scale=.2]
 \node[circle,draw,fill,inner sep=.5pt] (a) at (0,0) {};
 \node[circle,draw,fill,inner sep=.5pt] (b) at (1,0) {};
 \draw (a) edge[bend left=25] (b);
 \draw (a) edge[bend right=25] (b);
 \draw (.2,-.5) -- (.8,.6);
\end{tikzpicture}}}

\newcommand{\notadp}{{
\begin{tikzpicture}[baseline=-.55ex,scale=.2, every loop/.style={}]
 \node[circle,draw,fill,inner sep=.5pt] (a) at (0,0) {};
 \draw (a) edge[loop] (a);
 \draw (-.2,-.2) -- (.2,.5);
\end{tikzpicture}}}

\theoremstyle{plain}
  \newtheorem{thm}{Theorem}
  
  \newtheorem{prop}{Proposition}
  
  \newtheorem{cor}[prop]{Corollary}
  \newtheorem{conjecture}{Conjecture}
  \newtheorem{lemma}{Lemma}
\theoremstyle{definition}
  
  \newtheorem{rem}{Remark}

\newcommand{\alg}[1]{\mathfrak{{#1}}}
\newcommand{\co}[2]{\left[{#1},{#2}\right]} 

\newcommand{\Hom}{\mathrm{Hom}}

\newcommand{\K}{{\mathbb{K}}}
\newcommand{\Z}{{\mathbb{Z}}}
\newcommand{\N}{{\mathbb{N}}}

\newcommand{\Graphs}{{\mathsf{Graphs}}}

\newcommand{\Gra}{{\mathsf{Gra}}}

\newcommand{\mF}{\mathcal{F}}
\newcommand{\mG}{\mathcal{G}}

\newcommand{\Tw}{\mathit{Tw}}
\newcommand{\Def}{\mathrm{Def}}

\newcommand{\FreeLie}{\mathrm{FreeLie}}

\newcommand{\Id}{\mathrm{Id}}

\newcommand{\Lie}{\mathsf{Lie}}

\newcommand{\hoLie}{\mathsf{hoLie}}

\newcommand{\fGC}{\mathrm{fGC}}

\newcommand{\Ass}{\mathsf{Assoc}}
\newcommand{\Com}{\mathsf{Com}}

\newcommand{\bpm}{\begin{pmatrix}}
\newcommand{\epm}{\end{pmatrix}}

\newcommand{\GC}{\mathrm{GC}}

\newcommand{\fGCcdu}{{}^*\fGCc}

\newcommand{\fGCc}{\mathrm{fGCc}}

\newcommand{\WG}{WG}

\newcommand{\bbS}{\mathbb{S}}

\DeclareMathOperator{\gr}{gr}

\newcommand{\gra}{\mathit{gra}}
\newcommand{\grt}{\alg {grt}}

\newcommand{\Poiss}{\mathsf{Poiss}}
\newcommand{\hoPoiss}{\Poiss_{\infty}}

\newcommand{\vdim}{\mathit {dim}}

\begin{document}
\title{Differentials on graph complexes}

\author{Anton Khoroshkin}
\address{1004 Faculty of mathematics HSE, 7 Vavilova street, Moscow, Russia, 115280}
\email{akhoroshkin@hse.ru}

\author{Thomas Willwacher}
\address{Institute of Mathematics\\ University of Zurich\\ Winterthurerstrasse 190 \\ 8057 Zurich, Switzerland}
\email{thomas.willwacher@math.uzh.ch}

\author{Marko \v Zivkovi\' c}
\address{Institute of Mathematics\\ University of Zurich\\ Winterthurerstrasse 190 \\ 8057 Zurich, Switzerland}
\email{the.zivac@gmail.com}

\thanks{
A.K. has been partially supported by RFBR grants 13-02-00478, 13-01-12401, 
by "The National Research University--Higher School of Economics" Academic Fund Program in 2013-2014,
research grant 14-01-0124, by Dynasty foundation and Simons-IUM fellowship.
T.W. and M.\v Z. have been partially supported by the Swiss National Science foundation, grant 200021\_150012, and the SwissMAP NCCR funded by the Swiss National Science foundation.
}


\begin{abstract}
We study the cohomology of complexes of ordinary (non-decorated) graphs, introduced by M. Kontsevich.
We construct spectral sequences converging to zero whose first page contains the graph cohomology. 
In particular, these series may be used to show the existence of an infinite series of previously unknown and provably non-trivial cohomology classes, and put constraints on the structure of the graph cohomology as a whole.
\end{abstract}

\maketitle

\section{Introduction}

Graph complexes are graded vector spaces of formal linear combinations of isomorphism classes of graphs, with the differential defined by edge contraction (or, dually, vertex splitting).
The graph cohomology is the cohomology of these complexes. 
Various versions of graph complexes exist, for various types of graphs: ribbon graphs \cite{penner}, ordinary graphs \cite{K3,K4,K5}, directed acyclic graphs \cite{gcor}, graphs with external legs \cite{AroneLV,Turchin2,Turchin3} etc. 
The various graph cohomology theories are arguably some of the most fascinating objects in homological algebra. 
They have an elementary and simple combinatorial definition, and each of them plays a central role in a subfield of homological algebra or algebraic topology.
Yet, we know almost nothing about what the graph cohomology actually is, and even conjectures about its global structure are currently out of reach. 

In this paper, we will consider the most basic graph complexes, formed by ordinary graphs, without external legs or other decoration, introduced by M. Kontsevich.
Due to sign and degree conventions these complexes come in versions $\GC_d$ (a precise definition will be given below), where $d$ ranges over integers. Physically, $\GC_d$ is formed by vacuum Feynman diagrams of a topological field theory in dimension $d$. Alternatively, $\GC_d$ governs the deformation theory of the $E_d$ operads in algebraic topology \cite{grt}.

Due to an additional grading by loop order, the graph complexes $\GC_d$ for various $d$ of the same parity are isomorphic, up to degree shifts. Hence it suffices to study one of these complexes for $d$ even and one for $d$ odd, most often one takes $\GC_2$ and $\GC_3$.
Our current knowledge about the cohomology $H(\GC_2)$ and $H(\GC_3)$ is summarized as follows.
\begin{itemize}
\item It has been shown \cite{grt} that $H^{\leq -1}(\GC_2)=0$, while $H^{0}(\GC_2)=\grt_1$ is the Grothendieck-Teichm\"uller Lie algebra. It has recently been shown by F. Brown \cite{brown} that $\grt_1$ contains a free Lie algebra generated by generators $\sigma_3,\sigma_5,\sigma_7,\dots$. The Deligne-Drinfeld-Ihara conjecture in fact states that $\grt_1$ is equal to this free Lie algebra, up to completion. It is furthermore another famous conjecture due to Kontsevich and Drinfeld that $H^1(\GC_2)=0$. 
Apart from the cohomology in degree 0, computer experiments \cite{barnatanmk} have shown that there are sporadic classes in degrees $\geq 3$.  The computer generated data for $H^{\bullet}(\GC_2)$ may be found in Table \ref{tbl:evencanceling} on page~\pageref{tbl:evencanceling} below.
\item The odd version of the graph cohomology $H(\GC_3)$ has also been widely studied for its role in knot theory and finite type invariants. The best understood degree is the top degree $H^{-3}(\GC_3)$, see, e.g., \cite{barnatan}.\footnote{In some other conventions the top degree is placed in degree 0, while we prefer -3 for reasons explained below.} $H^{-3}(\GC_3)$ can be shown to be a commutative algebra. It is known that many non-trivial classes exist in $H^{-3}(\GC_3)$, provided by Chern-Simons theory.
It is a conjecture due to P. Vogel \cite{vogel} building upon work by J. Kneissler \cite{kneissler2,kneissler} that $H^{-3}(\GC_3)$ is generated as a commutative algebra by certain elements $t,\omega_0, \omega_1, \dots$, with relations $\omega_p\omega_q=\omega_0\omega_{p+q}$ and $P=0$, for a single (explicitly known) polynomial $P$.
Apart from the "dominant" degree $-3$ computer experiments \cite{barnatanmk} have shown that there are sporadic classes in degree $-6$.
The computer generated data for $H^{\bullet}(\GC_3)$ may be found in Table \ref{tbl:oddcanceling} on page~\pageref{tbl:oddcanceling} below.
\item There are upper and lower degree bounds on the graph cohomology in each loop order. Concretely, one has $H^{\leq -1}(\GC_2)=0$, and in loop order $l$ the cohomology of degrees $\geq l-2$ vanishes. Similarly, one can check that $H^{\geq -2}(\GC_3)=0$ and that in loop order $l$ the cohomology of degrees $\leq -l-2$ vanishes. 
\end{itemize}

The main result of the present paper is as follows. We choose to formulate it for $\GC_0$ and $\GC_1$ instead of $\GC_2$ and $\GC_3$ for degree issues. However, note that since $\GC_d$ is isomorphic to $\GC_{d+2}$ up to degree shifts, the result applies equally well to each even or odd $d$, again up to some degree shifts.
\begin{thm}\label{thm:main}
For each $d=0,1$, there is a spectral sequence, whose first page is
\[
E^1 = H(\GC_d) \oplus \prod_{\substack{k\geq 3 \\ k\equiv 2d+1 \text{ mod }4}} L_k
\]
with $L_k$ (the "loop classes") living in degree $k-d$, converging to $0$ for $d=0$ and to $\K[-1]$ for $d=1$. 
\end{thm}
In an enlarged complex $\GC_d^2$, the loop classes $L_k$ may be represented by graphs 
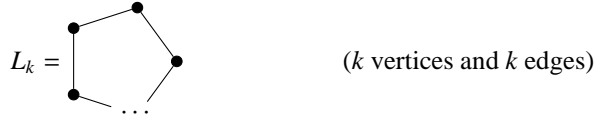
\[
L_k
=
\begin{tikzpicture}[baseline=-.65ex, scale=.75]
\node[int] (v1) at (0:1) {};
\node[int] (v2) at (72:1) {};
\node[int] (v3) at (144:1) {};
\node[int] (v4) at (216:1) {};
\node (v5) at (-72:1) {$\cdots$};
\draw (v1) edge (v2) edge (v5) (v3) edge (v2) edge (v4) (v4) edge (v5);
\end{tikzpicture}
\quad \quad \quad \quad \quad \quad \text{($k$ vertices and $k$ edges)}
\]

Theorem \ref{thm:main} in particular implies that classes in the graph cohomology come ``in pairs''\footnote{Note that we are a little sloppy here. Of course, there are not actually pairs in the sense of two well defined elements, but rather pairs of degree/loop order combinations between which a non-trivial differential exists on some page of the spectral sequence.}, with one class killing the other on some page of the spectral sequence. 
Curiously, we show in Proposition \ref{prop:wheelkills} below that the "partners" of the loop graphs $L_{4k+1}$ in $H(\GC_2)$ are exactly the classes $\sigma_{2k+1}$ conjecturally generating $\grt_1=H^0(\GC_2)$.
We conjecture that the "partners" of the loop graphs $L_{4k-1}$ in $H(\GC_3)$ are similarly variants of the conjectural generators of $H^{-3}(\GC_3)$ from above, cf. Conjecture \ref{conj:wheelkills}.
Degree considerations show that the "partners" of the remaining classes in the dominant degrees $H^0(\GC_2)$ and $H^{-3}(\GC_3)$ have to live in degrees $\geq 1$ or, respectively, degrees $\leq -4$.
We hence arrive at the following corollary, which is the first result to establish the existence of an infinite family of provably non-trivial graph cohomology classes in degrees other then the dominant ones.
 
\begin{cor}
The classes in $H^{\geq 1}(\GC_2)$ span a vector space of infinite dimension. More concretely, the "partners"\footnote{notation as above} of a basis of 
\[
\FreeLie(\sigma_3, \sigma_5,\sigma_7,\dots) / \mathit{span}(\sigma_3, \sigma_5,\dots) \subset H^0(\GC_2) / \mathit{span}(\sigma_3, \sigma_5,\dots)
\]
form a linearly independent set of infinite cardinality in $H^{\geq 1}(\GC_2)$.

Similarly, the classes in $H^{\leq -4}(\GC_3)$ span an infinite dimensional vector space.
\end{cor}

{\bf Remark:} The mechanism sketched above, i.e., one dominant degree of conjecturally known form, plus the "partners" in the spectral sequence yields, for the first time, a picture of the graph cohomology explaining all known classes in the computer accessible regime.
We note however that at high loop orders this mechanism does not provide enough classes to account for the Euler characteristic computations of \cite{eulerchar}. Very interestingly, and to a complete surprise of the authors, these Euler characteristics ``grow too fast'' in high loop order and, to an even larger surprise, are (again at high loop order) almost identical in the even and odd $d$ cases. This suggests that on top of the mechanism above, there is an additional source of classes, and this source is the same in the even and odd $d$ cases. 
This is completely unexpected, since there is otherwise no known link between $H(\GC_2)$ and $H(\GC_3)$.

Finally, let us remark that even though Theorem \ref{thm:main} states the cases of even and odd $d$ together, the nature and construction of the spectral sequences leading to the result are drastically different in the two cases. Hence, the treatment of even and odd $d$ is done in separate sections each, and completely separate proofs have to be given.

\begin{rem}
This paper is the first in a series of two papers.
 The second paper will discuss the extension of the additional differentials introduced in the present work to the ``hairy'' graph complexes studied in \cite{LambrechtsTurchin, Turchin1}.
\end{rem}

\subsection{Structure of the paper}
In section \ref{sec:grcomplexes} we recall the definitions of the graph complexes we study in this paper.
Sections \ref{sec:extradiffev} and \ref{sec:extradiffodd} contain a derivation of the spectral sequences mentioned in the introduction and the proof of Theorem \ref{thm:main}, separately for the cases of even and odd $d$.

In Appendix \ref{app:altproof} we sketch an alternative and more conceptual argument for the main Theorem in the case of odd $d$, using the rigidity of the associative operad.
Appendix \ref{app:7loop} contains an explicit calculation of the graph cohomology class cancelled by the loop graph of length 7 in the spectral sequence. Finally, in Appendix \ref{app:specseq} we collect some remarks and auxiliary results about the convergence of spectral sequences used elsewhere in the paper.

\subsection*{Acknowledgements}
We thank B. Feigin, N. Markarian and S. Merkulov for stimulating discussions.

\section*{Notation}
We work over a fixed ground field $\K$ of characteristic zero. All vector spaces or graded vector spaces are taken over this ground field.
The vector space freely generated by a set $S$ we denote by $\K\langle S\rangle$.
The phrase \emph{differential graded} will be abbreviated to \emph{dg} as usual.
In general, we work in cohomological conventions, i.e., the differentials will have degree 1 unless otherwise stated.
The symmetric groups will be denoted by $S_n$.

\section{Graph complexes}\label{sec:grcomplexes}

Let $\gra_{N,k}$ be the set of directed graphs with vertex set $\{1,\dots,N\}$ and directed edge set $\{1,\dots, k\}$. This set carries a natural right action of the group $S_N\times S_k\ltimes (S_2)^k$, with $S_N$ acting by relabeling the vertices, $S_k$ by relabeling the edges and the $S_2$ factors by reversing the direction of the edges.
One may build right $S_N$ modules
\[
 \Gra_n(N) = \prod_{k} \left( \K\langle \gra_{N,k} \rangle \otimes \K[n-1]^{\otimes k}\right)_{S_k\ltimes (S_2)^k}
\]
where the action of $S_k$ is induced by the action on $\gra_{N,k}$ if $n$ is odd, and defined with an extra sign if $n$ is even, and the $S_2$ actions are induced by the actions on $\gra_{N,k}$ if $n$ is even and defined with an extra sign if $n$ is odd.
The spaces $\Gra_n(N)$ assemble into an operad $\Gra_n$. The operadic composition is by inserting one graph at a vertex of another and reconnecting the dangling edges in all possible ways, cf. the introductory sections of \cite{grt}. 

The total space of any operad, and the invariants or coinvariants under the symmetric group actions thereof form a dg Lie algebra, cf. \cite{lodayval}. In particular, the spaces 
\[
 \fGC_n := \prod_N \Gra_n\{-n\}(N)_{S_N}
\]
form a graded Lie algebra. Its elements are series of graphs with unidentifiable vertices, i.e., of isomorphism classes of graphs.
The Lie bracket of two homogeneous elements $\gamma_1\in \Gra_n\{-n\}(N_1)_{S_{N_1}}\subset \fGC_n$ and $\gamma_2\in \Gra_n\{-n\}(N_2)_{S_{N_2}}\subset \fGC_n$ is 
\[
 [\gamma_1,\gamma_2] = \gamma_1\bullet\gamma_2 - (-1)^{|\gamma_1||\gamma_2|}\gamma_2\bullet \gamma_1
\]
where the pre-Lie product $\bullet$ is defined such that 
\[
 \gamma_1\bullet\gamma_2 = \sum_{j=1}^{N_1} \pm \gamma_1 \circ_j \gamma_2
  = \sum_{x\in V(\gamma_1)}
\begin{tikzpicture}[baseline=-.65ex]
 \node (v) at (-.4,0) {$\gamma_1$};
 \node[ext] (w) at (.4,0) {$\gamma_2$};
 \node at (.6,-.2) {$\scriptstyle x$};
 \draw (v) edge (w);
 \draw (-.2,-.1) edge (w);
 \draw (-.2,.1) edge (w);
\end{tikzpicture}.
\]
Here $V(\gamma_1)$ is the set of vertices in graphs in $\gamma_1$, and every vertex $x\in V(\gamma_1)$ is replaced by the whole graph $\gamma_2$ and 
\begin{tikzpicture}[baseline=-.65ex]
 \node (v) at (-.4,0) {$\gamma_1$};
 \node[ext] (w) at (.4,0) {$\gamma_2$};
 \node at (.6,-.2) {$\scriptstyle x$};
 \draw (v) edge (w);
 \draw (-.2,-.1) edge (w);
 \draw (-.2,.1) edge (w);
\end{tikzpicture}
denotes the sum of all possible ways edges from $\gamma_1$ connected to $x$ can be reconnected to vertices of $\gamma_2$. 

In pictures we will denote elements of $\fGC_n$ by graphs with black (unidentifiable) vertices.
There is a Maurer-Cartan element 
\[
m=
 \begin{tikzpicture}[baseline=-.65ex]
  \node[int] (v) at (0,0) {};
  \node[int] (w) at (0.5,0) {};
  \draw (v) edge (w);
 \end{tikzpicture},
\]
i.e., $[m,m]=0$ and $|m|=1$. (It stems from the operad map $\Lie\{n-1\}\to \Gra_n$.) We will consider $\fGC_n$ as a complex with differential 
\begin{equation}\label{equ:deltadef}
 \delta=[m, \cdot].
\end{equation}


We also define several dg Lie subalgebra of $\fGC_n$. The linear span of the connected graphs we denote by $\fGCc_n\subset \fGC_n$.
The graph complex $\GC_n$, as originally defined by Kontsevich, is the dg Lie subalgebra $\GC_n\subset \fGCc_n$ spanned by the graphs all of whose vertices are at least trivalent.

Since the signs in the above definition only depend on the parity of $n$, the graph complexes $\fGC_n$ for various $n$ of the same parity are isomorphic, up to unimportant shifts of degrees. Hence, to obtain a complete understanding, it suffices to consider $\fGC_n$ for one even and one odd $n$.
In the following sections we will consider the cases $n=0,1$. In the first (even) case the (cohomological) degree of vertices is 0, and that of the edges is -1. In the second (odd) case the degree of the vertices is 1, while that of the edges is 0.

%

\section{The extra differential (even case)}\label{sec:extradiffev}

We define an additional differential on $\fGC_0$ by the Lie bracket with the ``tadpole'' graph
\[
 \nabla=[
\begin{tikzpicture}[baseline=-.65ex, scale=.5, every loop/.style={}]
 \node[int] (v) at (0,0){};
\draw (v) edge[loop] (v);
\end{tikzpicture}
,\cdot].
\]
The effect of $\nabla$ is to add an additional edge, in all possible ways. One easily checks that $\delta+\nabla$ is also a differential, because $\delta\nabla+\nabla\delta=0$ and $\nabla^2=0$.

In this section we consider the sub-complex $\fGC_0^\notadp\subset\fGC_0$ spanned by graphs without tadpoles. It is a sub-complex with respect to both differentials $\delta$ and $\delta+\nabla$. On the level of cohomology (with respect to $\delta$) $H(\fGC_0^\notadp)$ and $H(\fGC_0)$ differ by only one class, the single vertex tadpole (see \cite[Proposition 3.4]{grt}). We will also consider the connected parts $\fGCc_0$ and $\fGCc_0^\notadp$.


\subsection{Cohomology of the extra differential}
\begin{prop}\label{prop:evenD}
 $H(\fGC_0^\notadp, \nabla)$ is one-dimensional, the class being represented by the single vertex graph.
\end{prop}
\begin{proof}
We have (by definition) $\fGC_0^\notadp=\prod_n (V_n)_{S_n}$, where $V_n$ is a graded vector space of linear combinations of graphs with $n$ numbered vertices.
The space $V_n$ is finite dimensional and a $\nabla$ acts on each $V_n$ separately.
Since taking (co)invariants with respect to a finite group action commutes with taking cohomology, it suffices to 
calculate $H(V_n, \nabla)$. For $n=1$ it is 1-dimensional. For $n\geq 2$ there is an explicit homotopy $h:V_n\to V_n$:
\[
 h \Gamma =
\begin{cases}
 \pm(\Gamma - \{1,2\}) & \text{if there is an edge between vertices 1 and 2} \\
0 & \text{otherwise},
\end{cases}
\]
where the sign is chosen $+$ if $(1,2)$ is the first edge in the implicit ordering. We find that $\nabla \circ h+h\circ\nabla=\mathit{id}_{V_n}$, and hence $H(V_n, \nabla)=0$.

We also remark that an alternative, "more symmetric" proof may be obtained by defining a homotopy $\tilde h$, such that 
\[
\tilde h\Gamma = \sum_e \pm (\Gamma - \{e\})
\] 
is the sum over all ways of removing one edge from the graph $\Gamma$. Then $(\nabla\tilde h+\tilde h\nabla)\Gamma=\frac {n(n-1)}{2}\Gamma$ where $n$ is the number of vertices, and hence we arrive at the same conclusion as before.
\end{proof}

\begin{cor}\label{cor:evenex}
 $H(\fGCc_0^\notadp, \nabla)$ is two-dimensional, the classes being represented by the single vertex graph and the graph with two vertices and one edge.
\end{cor}
\begin{proof}

It suffices to consider the sub-complex $\fGCc_0^n$ spanned by graphs with $n$ vertices.
Clearly $\fGCc_0^1$ and $\fGCc_0^2$ are 1 dimensional and hence the differential acts trivially.
We assume inductively that for all $3\leq j\leq n-1$ we have $H(\fGCc_0^j)=0$. Let $\gamma$ be a cocycle in $\fGCc_0^n$, i.e.\ $\nabla\gamma=0$. Let $\nu$ be an element of $\fGC_0^n$, such that $\nabla\nu=\gamma$.
We decompose $\nu$ into
\[
 \nu =\nu_1+\nu_2+\cdots + \nu_k.
\]
where $\nu_l$ is the term of graphs with $l$ connected components. We chose $\nu$ such that $k$ is least possible. If $k=1$, $\nu$ is connected and $\gamma$ is a exact in $\fGCc_0^n$ what we need. So let us assume $k>1$.

We split $\nabla=\nabla_0+\nabla_1$ on $\fGC_0^n$ into a term that leaves the number of connected components invariant, and one that lowers the number of connected components (necessarily by one).
The leading term $\nu_k$ must be $\nabla_0$-closed. Since $\nabla_0$ does not change the number of connected components it commutes with the action of symmetry group that permutes the components, and we can distinguish them. So, every connected component in $\nu_k$ is $\nabla_0$- (that is now $\nabla$-) closed. If one of them is exact we may replace it with its preimage and get $\nu'_k$. Then $\nu_k=\nabla_0\nu'_k$ is a ($\nabla_0$-)coboundary and we may replace $\nu$ with the cohomologous cycle $\nu-\nabla \nu'_k$ that does not have a $\nu_k$ part, contradicting the minimality of $k$. So every connected component represents a nontrivial cohomology class, i.e., by the induction hypothesis, every connected component is an isolated vertex or an isolated edge between two vertices. Therefore, and because of symmetry reasons, $\nu_k$ is a linear combination of graphs which are a union of isolated vertices, and possibly at most one isolated edge.
If there is an isolated edge in $\nu_k$, for $\nu'$ being the union of all isolated vertices, $\nu$ can be replaced with $\nu-\nabla \nu'$, that again does not have a $\nu_k$ part contradicting the minimality of $k$. So there may only be $n$ isolated vertices and $k=n$. But the equality $\nabla_1\nu_n+\nabla_0\nu_{n-1}=0$ can be satisfied only if $\nu_n=0$. Hence again $\nu_k=0$, concluding the proof.

\end{proof}

\subsection{A picture of the even graph cohomology}
\begin{cor}\label{cor:even}
There is a spectral sequence converging to
\[
 H(\fGCc_0^\notadp, \delta+\nabla)=0
\]
whose $E^1$ term is 
\[
 (H(\fGCc_0^\notadp,\delta), \nabla).
\]
\end{cor}
\begin{proof}
First consider the descending exhaustive filtration of $(\fGCc_0, \delta+\nabla)$
\[
 \mF^1=\fGCc_0^\notadp \supset \mF^2 \supset \cdots
\]
where $\mF^p$ consists of graphs with $\geq p$ vertices, and set up a spectral sequence associated to that filtration. On the $E^1$ page we have $H(\fGCc_0^\notadp,\nabla)$, which is generated by two generators, the single vertex graph and the graph with two connected vertices (cf. Corollary \ref{cor:evenex}). The differential on this page is induced by $\delta$, mapping the single-vertex graph to the graph with two vertices. Hence the next page in the spectral sequence is $E^2=0$.
It follows by a standard spectral sequence argument (Proposition \ref{prop:App1} from the Appendix) that $H(\fGCc_0^\notadp, \delta+\nabla)=0$.

Next consider the exhaustive descending filtration
\[
 \mG^{0}=\fGCc_0^\notadp \supset \mG^{1}\supset \cdots
\]
where $\mG^{p}$ is spanned by graphs whose Betti number $e-v$ is at least $p-1$, and set up a spectral sequence associated to that filtration. On the $E^1$ page we have $H(\fGCc_0^\notadp,\delta)$. This spectral sequence indeed converges to the cohomology (again by Proposition \ref{prop:App1}), i.~e., to 0.
\end{proof}

The table \ref{tbl:evencanceling} represents the $E^1$ page of the spectral sequence from Corollary \ref{cor:even}, i.e.\ $H(\fGCc_0^\notadp,\delta)$. The column number represents the number of edges $e$ and the row number represents the Betti number minus one $b=e-v$, and the displayed numbers are the dimensions of the respective parts of the graph cohomology $H(\fGCc_0^\notadp,\delta)$.\footnote{The numbers in the table have been partially taken from \cite{barnatanmk}, and have partially been computed by the second author (unpublished). The latter computations have been performed in floating point arithmetic (due to limited computer power) and hence are not mathematically rigorous. }

\begin{table}
\begin{tikzpicture}
\matrix (mag) [matrix of nodes,ampersand replacement=\&]
{
{}\& 0 \& 1 \& 2 \& 3 \& 4 \& 5 \& 6 \& 7 \& 8 \& 9 \&10 \&11 \&12 \&13 \&14 \&15 \&16 \&17 \&18 \&19 \&20 \&21 \&22 \&23 \&24 \&25 \&26 \&27 \\
-1\& 0 \& 0 \& 0 \& 0 \& 0 \& 0 \& 0 \& 0 \& 0 \& 0 \& 0 \& 0 \& 0 \& 0 \& 0 \& 0 \& 0 \& 0 \& 0 \& 0 \& 0 \& 0 \& 0 \& 0 \& 0 \& 0 \& 0 \& 0 \\
0 \&   \& 0 \& 0 \& 0 \& 0 \& 1 \& 0 \& 0 \& 0 \& 1 \& 0 \& 0 \& 0 \& 1 \& 0 \& 0 \& 0 \& 1 \& 0 \& 0 \& 0 \& 1 \& 0 \& 0 \& 0 \& 1 \& 0 \& 0 \\
1 \&   \&   \& 0 \& 0 \& 0 \& 0 \& 0 \& 0 \& 0 \& 0 \& 0 \& 0 \& 0 \& 0 \& 0 \& 0 \& 0 \& 0 \& 0 \& 0 \& 0 \& 0 \& 0 \& 0 \& 0 \& 0 \& 0 \& 0 \\
2 \&   \&   \&   \& 0 \& 0 \& 0 \& 1 \& 0 \& 0 \& 0 \& 0 \& 0 \& 0 \& 0 \& 0 \& 0 \& 0 \& 0 \& 0 \& 0 \& 0 \& 0 \& 0 \& 0 \& 0 \& 0 \& 0 \& 0 \\
3 \&   \&   \&   \&   \& 0 \& 0 \& 0 \& 0 \& 0 \& 0 \& 0 \& 0 \& 0 \& 0 \& 0 \& 0 \& 0 \& 0 \& 0 \& 0 \& 0 \& 0 \& 0 \& 0 \& 0 \& 0 \& 0 \& 0 \\
4 \&   \&   \&   \&   \&   \& 0 \& 0 \& 0 \& 0 \& 0 \& 1 \& 0 \& 0 \& 0 \& 0 \& 0 \& 0 \& 0 \& 0 \& 0 \& 0 \& 0 \& 0 \& 0 \& 0 \& 0 \& 0 \& 0 \\
5 \&   \&   \&   \&   \&   \&   \& 0 \& 0 \& 0 \& 0 \& 0 \& 0 \& 0 \& 0 \& 0 \& 1 \& 0 \& 0 \& 0 \& 0 \& 0 \& 0 \& 0 \& 0 \& 0 \& 0 \& 0 \& 0 \\
6 \&   \&   \&   \&   \&   \&   \&   \& 0 \& 0 \& 0 \& 0 \& 0 \& 0 \& 0 \& 1 \& 0 \& 0 \& 0 \& 0 \& 0 \& 0 \& 0 \& 0 \& 0 \& 0 \& 0 \& 0 \& 0 \\
7 \&   \&   \&   \&   \&   \&   \&   \&   \& 0 \& 0 \& 0 \& 0 \& 0 \& 0 \& 0 \& 0 \& 1 \& 0 \& 0 \& 1 \& 0 \& 0 \& 0 \& 0 \& 0 \& 0 \& 0 \& 0 \\
8 \&   \&   \&   \&   \&   \&   \&   \&   \&   \& 0 \& 0 \& 0 \& 0 \& 0 \& 0 \& 0 \& 0 \& 0 \& 1 \& 0 \& 0 \& 1 \& 0 \& 0 \& 0 \& 0 \& 0 \& 0 \\
9 \&   \&   \&   \&   \&   \&   \&   \&   \&   \&   \& 0 \& 0 \& 0 \& 0 \& 0 \& 0 \& 0 \& 0 \& 0 \& 0 \& 1 \& 0 \& 0 \& 2 \& 0 \& 0 \& 0 \& 1 \\
10\&   \&   \&   \&   \&   \&   \&   \&   \&   \&   \&   \& 0 \& 0 \& 0 \& 0 \& 0 \& 0 \& 0 \& 0 \& 0 \& 0 \& 0 \& 2 \& 0 \& ? \& ? \& ? \& ? \\
11\&   \&   \&   \&   \&   \&   \&   \&   \&   \&   \&   \&   \& 0 \& 0 \& 0 \& 0 \& 0 \& 0 \& 0 \& 0 \& 0 \& 0 \& 0 \& 0 \& 2 \& ? \& ? \& ? \\
12\&   \&   \&   \&   \&   \&   \&   \&   \&   \&   \&   \&   \&   \& 0 \& 0 \& 0 \& 0 \& 0 \& 0 \& 0 \& 0 \& 0 \& 0 \& 0 \& 0 \& 0 \& 3 \& ? \\
};
\draw (mag-2-1.north west) -- (mag-1-29.south east);
\draw (mag-1-2.north west) -- (mag-15-1.south east);
\draw[-latex, thick] (mag-3-7) edge (mag-5-8);
\draw[-latex, thick] (mag-3-11) edge (mag-7-12);
\draw[-latex, thick] (mag-3-15) edge (mag-9-16);
\draw[-latex, thick] (mag-8-17) edge (mag-10-18);
\draw[-latex, thick] (mag-3-19) edge (mag-11-20);
\draw[-latex, thick] (mag-10-21) edge (mag-12-22);
\draw[-latex, thick] (mag-3-23) edge (mag-13-24);
\draw[-latex, thick] (mag-11-23) edge (mag-13-24);
\draw[-latex, thick] (mag-12-25) edge (mag-14-26);
\draw[-latex, thick] (mag-3-27) edge (mag-15-28);
\end{tikzpicture}
\caption{\label{tbl:evencanceling} Table of cohomology of $\fGCc_0$. The arrows indicate some cancellations in the spectral sequence.}
\end{table}

The classes for $b=0$ are represented by the loop graphs $L_k$ (see Figure \ref{fig:loops}). Only the elements $L_{4j+1}$ are non-zero in $\fGCc_0$ by symmetry reasons, and $L_1\notin\fGCc_0^\notadp$ by definition.
Corollary \ref{cor:even} implies that classes in the graph cohomology occur in pairs, with one partner cancelling the other on some page of the above spectral sequence. Some cancellations are indicated by arrows in the table above.

In particular, one may ask which classes are cancelled by the graph cohomology classes $L_{4j+1}$ on some page of the spectral sequence. Note that $H^0(\fGCc_2)\cong \grt_1$ is the Grothendieck-Teichm\"uller Lie algebra \cite{grt} and that this Lie algebra contains a free Lie subalgebra generated by elements $\sigma_3,\sigma_5,\cdots\in \grt_1$, see \cite{brown}. Conjecturally, $\grt_1$ is equal to this free Lie algebra. Explicit integral formulas for the graph cocycle corresponding to $\sigma_{2j+1}$, which we also call $\sigma_{2j+1}$ abusing notation, are derived in \cite{carlothomas}. We may identify $H(\fGCc_2)$ with $H(\fGCc_0)$ up to degree shifts, and by slight abuse of notation we will the $\sigma_{2j+1}$ both as elements of $H(\fGCc_2)$ and $H(\fGCc_0)$. We then have the following result.

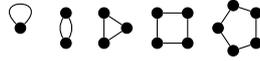
\begin{figure}
 \[
\begin{tikzpicture}
\node (v0) at (0.5,1) [int] {};
\draw (v0) to [out=45,in=135,loop] ();
\node (v1) at (1.1,0.8) [int] {};
\node (v2) at (1.1,1.2) [int] {};
\draw (v1) edge [bend left] (v2);
\draw (v1) edge [bend right] (v2);
\node (v3) at (1.6,0.8) [int] {};
\node (v4) at (1.6,1.2) [int] {};
\node (v5) at (1.9,1) [int] {};
\node (v6) at (2.3,0.8) [int] {};
\node (v7) at (2.3,1.2) [int] {};
\node (v8) at (2.7,1.2) [int] {};
\node (v9) at (2.7,0.8) [int] {};
\node (v10) at (3.1,1) [int] {};
\node (v11) at (3.3,1.3) [int] {};
\node (v12) at (3.3,0.7) [int] {};
\node (v13) at (3.6,0.8) [int] {};
\node (v14) at (3.6,1.2) [int] {};
\draw (v3)--(v4) (v4)--(v5) (v5)--(v3) (v6)--(v7) (v7)--(v8) (v8)--(v9) (v9)--(v6) (v10)--(v11) (v11)--(v14) (v14)--(v13) (v13)--(v12) (v12)--(v10) ;
\end{tikzpicture}
\]
\caption{\label{fig:loops} The loop graph cocycles $L_k$. Note that only the $L_{4j+1}$ are non-zero in the graph complex $\fGCc_0$.}
\end{figure}

\begin{prop}\label{prop:wheelkills}
 In the spectral sequence of Corollary \ref{cor:even} the graph cocycle $\sigma_{2j+1}$ and the loop cocycle $L_{4j+1}$ (see Figure \ref{fig:loops}) survive up to the $E^{2j}$ page, where a multiple of $L_{4j+1}$ cancels $\sigma_{2j+1}$.
\end{prop}
For $j=1$, the ``process'' is illustrated in Figure \ref{fig:nablaprocess}.

\begin{figure}
\[
\begin{tikzpicture}[baseline=-.65ex]
 \node (a) at (-.1,0) {$L_5$=
 \begin{tikzpicture}[baseline=-.65ex]
  \node[int] (v1) at (0:.5) {};
  \node[int] (v2) at (72:.5) {};
  \node[int] (v3) at (144:.5) {};
  \node[int] (v4) at (216:.5) {};
  \node[int] (v5) at (288:.5) {};
  \draw (v1) edge (v2) edge (v5) (v3) edge (v2) edge (v4) (v4) edge (v5);
 \end{tikzpicture}};
 \node (b) at (2.7,-2) {
\begin{tikzpicture}[baseline=-.65ex]
  \node[int] (v1) at (0:.5) {};
  \node[int] (v2) at (72:.5) {};
  \node[int] (v3) at (144:.5) {};
  \node[int] (v4) at (216:.5) {};
  \node[int] (v5) at (288:.5) {};
  \draw (v1) edge (v2) edge (v5) edge (v4) (v3) edge (v2) edge (v4) (v4) edge (v5);
 \end{tikzpicture}};
 \node (c) at (0.3,-2) {
 \begin{tikzpicture}[baseline=-.65ex]
  \node[int] (v1) at (0:.5) {};
  \node[int] (v2) at (90:.5) {};
  \node[int] (v3) at (180:.5) {};
  \node[int] (v4) at (270:.5) {};
  \draw (v1) edge (v2) edge (v3) edge (v4) (v3) edge (v2) edge (v4);
 \end{tikzpicture}};
 \node (d) at (3.1,-4) {
 \begin{tikzpicture}[baseline=-.65ex]
  \node[int] (v1) at (0:.5) {};
  \node[int] (v2) at (90:.5) {};
  \node[int] (v3) at (180:.5) {};
  \node[int] (v4) at (270:.5) {};
  \draw (v1) edge (v2) edge (v3) edge (v4) (v3) edge (v2) edge (v4) (v2) edge (v4);
 \end{tikzpicture}$\;=\sigma_3$};
 \draw (a) edge[|->] node[auto] {$\nabla$} (b);
 \draw (c) edge[|->] node[auto] {$\delta$} (b);
 \draw (c) edge[|->] node[auto] {$\nabla$} (d);
\end{tikzpicture}
\] 
\caption{\label{fig:nablaprocess} The loop cocyle $L_5$ cancels the graph cocyle $\sigma_{3}$ in the spectral sequence, cf. Proposition \ref{prop:wheelkills}. Prefactors are omitted.}
\end{figure}
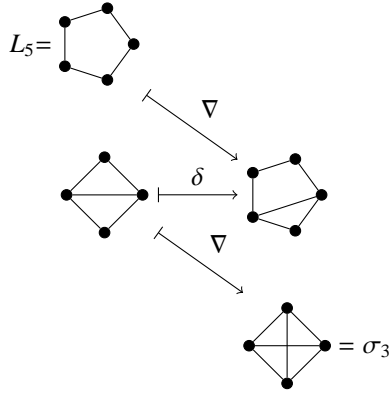

\begin{proof}
Consider the wheel graphs 
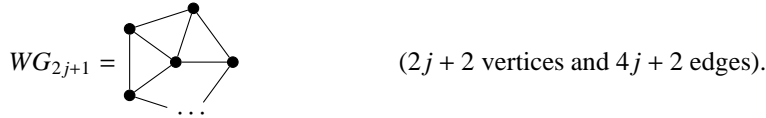
\[
\WG_{2j+1}
=
\begin{tikzpicture}[baseline=-.65ex, scale=.75]
\node[int] (c) at (0,0){};
\node[int] (v1) at (0:1) {};
\node[int] (v2) at (72:1) {};
\node[int] (v3) at (144:1) {};
\node[int] (v4) at (216:1) {};
\node (v5) at (-72:1) {$\cdots$};
\draw (v1) edge (v2) edge (v5) (v3) edge (v2) edge (v4) (v4) edge (v5)
      (c) edge (v1) edge (v2) edge (v3) edge (v4);
\end{tikzpicture}
\quad \quad \quad \quad \quad \quad \text{($2j+2$ vertices and $4j+2$ edges)}.
\]
It has been shown in \cite{grt} that any representative of the class $\sigma_{2j+1}$ contains the graph $\WG_{2j+1}$ with non-trivial coefficient.
We will show that any $2j+2$-vertex graph cocyle which has a non-zero coefficient in front of the wheel graph $\WG_{2j+1}$ will survive up to the $2j$-th page of the spectral sequence, where it is killed by a multiple of the loop graph $L_{4j+1}$.

More precisely, we will show the dual statement: 
Consider the complex $\fGCcdu_0$ (pre-)dual to $\fGCc_0$. It carries a differential $d+\nabla$, where $d$ acts as edge contraction and $\nabla$ deletes one edge.\footnote{In fact, this $\nabla$ is combinatorially the same operator as the homotopy $\tilde h$ occurring in the proof of Proposition \ref{prop:evenD}.} Dually to the above situation there is a spectral sequence on $\fGCcdu_0$ converging to $0$, with $H(\fGCcdu_0,d)$ on the first page.
We will show that the class of the graph cocyle $\WG_{2j+1}$ survives in this spectral sequence up to the $2j$-th page, where it kills (a multiple of) the class of the loop graph $L_{4j+1}$. We leave it to the reader to check that this dual statement implies the primal statement, i.~e., the statement of Proposition \ref{prop:wheelkills}.

Consider the degree 0 cocycle $X_1=\WG_{2j+1}$ as a class on the first page of the spectral sequence.
First note: (i) This class cannot be killed on some later page of the spectral sequence by one in lower degree since there is no cohomology of suitable degree and Betti number in the graph complex.\footnote{This follows from the fact that $H^{\leq -1}(\GC_2)=0$ shown in \cite{grt}.} (ii) Classes potentially killed by $X_1$ on the $k$-th page of the spectral sequence must be made of graphs with $2j+1+k$ vertices and $4j+1$ edges. (iii) The differential on the first page maps $X_1$ to (the class represented by) $X_1'=\nabla \WG_{2j+1}$ which lives in the subcomplex (say $C$) of $(\fGCcdu_0, d)$ spanned by graphs containing at least one bivalent vertex. Since $\nabla$ can only decrease the valence of edges, it follows that the class to be killed by $X_1$ will live in a subquotient of the cohomology $H(C,d)$. But the complex $C$ has only one dimensional cohomology with the relevent number of edges, represented by $L_{4j+1}$.
Hence the class of $L_{4j+1}$ must be the class killed by $X_1$, and by (i) this must happen on the $k=2j$-th page of the spectral sequence.
\end{proof}

There are many more elements in $\grt_1\cong H^0(\fGCc_2)$ than the generators of course, at least all commutators $[\grt_1,\grt_1]$. For each of these elements there has to be a ``partner'' cohomology class in the graph complex, that kills it at some stage of the spectral sequence. This partner must live in positive cohomological degrees, and can hence in particular not be another element of $[\grt_1,\grt_1]$. Hence we arrive at the following Corollary of Corollary \ref{cor:even}.

\begin{cor}\label{cor:evendimcor}
There are infinitely many graph cohomology classes of positive degrees in $H(\GC_2)$. There are inequalities
 \[
  \sum_{j\geq 1} \vdim H^{2j-1}(\fGCc_2)_{\beta+j} \geq \vdim_{\beta} \FreeLie(\sigma_3,\sigma_5,\cdots)
 \]
 where $H^{2j-1}(\fGCc_2)_{\beta+j}$ is the piece of the graph cohomology of Betti number $\beta+j$ and $\vdim_{\beta} \FreeLie(\sigma_3,\sigma_5,\cdots)$ is the dimension of the degree $\beta$ subspace of the free Lie algebra, where the generator $\sigma_{2j+1}$ carries degree $2j+1$.
\end{cor}

The corollary indicates where to expect graph cohomology classes. However, it does not answer two questions:
\begin{itemize}
 \item In what bidegree (cohomological and Betti number) do the extra cohomology classes cancelling $\grt_1$ elements live? 
 \item Are there any other graph cohomology classes or not, i.~e., are the inequalities in Corollary \ref{cor:evendimcor} strict?
\end{itemize}

For the first question one may at least formulate a guess. Namely, since the graph complex is a dg Lie algebra, we may produce a lot of cocycles in degree $4j-1$ (with $j\geq 1$) by acting repeatedly with the $\grt_1\cong H^0(\fGCc)$ on the loop cocyle $L_{4j+1}$. 
Furthermore note that 
\begin{itemize}
 \item The extra differential $\nabla$ is a derivation with respect to the Lie bracket.
 \item The cocycles $\sigma_{2j+1}$ may be chosen $\nabla$-closed, see \cite{carlothomas}.
\end{itemize}
Hence one can check that cocycles obtained by taking brackets of $W_{4j+1}$ with some $\sigma_{2k+1}$'s repeatedly will (-if they survive long enough-) cancel elements of $\grt_1$ on the $2j$-th page of the spectral sequence. For example, $[\sigma_5, L_5]$ cancels $[\sigma_5,\sigma_3]$ on the second page. This sets an upper bound on when (i.e., on what page of the spectral sequence) certain Lie words in $\grt_1$ are killed. However, note that not every Lie bracket of $\grt_1$ elements with loop classes is a non-trivial class, for example $[\sigma_3, L_9]$ (which would kill $[\sigma_5, \sigma_3]$ on the $4$th page of the spectral sequence if it survived) is exact.

Nevertheless one can expect a lot of classes in degrees $4j-1$, $j=1,2,\dots$. In fact, all positive degree classes that have been found numerically up to the present day live in these degrees, with all other positive degrees unoccupied.
It hence seems tempting to raise the conjecture that the even graph cohomology is concentrated in cohomological degrees $0,3,7,11,15,\dots$. However, the Euler characteristic computations of \cite{eulerchar} together with new numerical experiments done by the authors seem to indicate that this is not true in high loop orders.

\section{The extra differential (odd case)}\label{sec:extradiffodd}

We consider the following degree 1 element of $\fGCc_1$
\begin{equation*}\label{equ:thetaMC}
 m = 
\sum_{j\geq 1}
\frac{1}{(2j+1)!} \,
\begin{tikzpicture}[baseline=-.65ex,scale=.8]
 \node[int] (v) at (0,.5) {};
 \node[int] (w) at (0,-0.5) {};
 \draw (v) edge[very thick] node[right] {$\scriptstyle 2j+1$} (w);
\end{tikzpicture},
\end{equation*}
where the thick line labeled by a number $2j+1$ represents a $2j+1$-fold edge, i.e., $2j+1$ edges connecting the same pair of vertices.

\begin{lemma}
 The element $m$ is a Maurer-Cartan element, i.~e., $\delta m + \frac{1}{2}\co{m}{m}=0$.
\end{lemma}
\begin{proof}
Let
$$
m' := m +
\begin{tikzpicture}[baseline=-.65ex, scale=.5]
 \draw (0,.5) node[int] {} -- (0,-.5) node[int] {};
\end{tikzpicture}
=
\sum_{j\geq 0}
\frac{1}{(2j+1)!} \;
\begin{tikzpicture}[baseline=-.65ex,scale=.8]
 \node[int] (v) at (0,.5) {};
 \node[int] (w) at (0,-0.5) {};
 \draw (v) edge[very thick] node[right] {$\scriptstyle 2j+1$} (w);
\end{tikzpicture}
$$
We have to check that $m'\bullet m'=0$. One checks that
 \[
\begin{tikzpicture}[baseline=-.65ex,scale=.8]
 \node[int] (v) at (0,.5) {};
 \node[int] (w) at (0,-0.5) {};
 \draw (v) edge[very thick] node[right] {$m$} (w);
\end{tikzpicture}
\bullet
\begin{tikzpicture}[baseline=-.65ex,scale=.8]
 \node[int] (v) at (0,.5) {};
 \node[int] (w) at (0,-0.5) {};
 \draw (v) edge[very thick] node[right] {$n$} (w);
\end{tikzpicture}
=
2
\sum_{i+j=m}
\binom{m}{i}\,
\begin{tikzpicture}[baseline=-.65ex,scale=.8]
 \node[int] (v) at (0,.5) {};
 \node[int] (w1) at (-.5,-0.5) {};
 \node[int] (w2) at (.5,-0.5) {};
 \draw (v)  edge[very thick] node[left] {$i$} (w1)
            edge[very thick] node[right] {$j$} (w2)
       (w1) edge[very thick] node[below] {$n$} (w2);
\end{tikzpicture}
\]
 where a fat edge labeled by an index $k$ stands for a $k$-fold edge as before.
 Let us denote the graphs on the left by $X_m$ and $X_n$ and the graph on the right by $Y_{i,j,n}$.
 By symmetry, we have the following identities
 \begin{align*}
  X_m &= (-1)^{m+1} X_m
  &
  Y_{i,j,n} &= (-1)^{n+1} Y_{j, i,n} = (-1)^{i+1}Y_{i,n,j}.
 \end{align*}
In particular $X_m=0$ for $m$ even.
One computes:
 \begin{align*}
  m'\bullet m' &=
  \sum_{k, l\geq 1,odd} \frac{1}{k! l!}
  X_k \bullet X_l
  =
  2
  \sum_{k, l\geq 1,odd}
  \sum_{i+j=k} \frac{1}{i! j! l!}
  Y_{i,j,l}
  \\&=
  4
  \sum_{i\geq 0,even}\sum_{j, l\geq 1,odd} \frac{1}{i! j! l!}
  Y_{i,j,l}
  =
  2
  \sum_{i\geq 0,even}\sum_{j, l\geq 1,odd} \frac{1}{i! j! l!}
  (Y_{i,j,l}+Y_{i,l,j})
  \\ &=0.  
 \end{align*}
\end{proof}

Hence we may compute the cohomology of $\fGCc_1$ with respect to the differential $\delta + \co{m}{\cdot} = \co{m'}{\cdot}$.
It turns out that this cohomology can in fact be computed completely.

\begin{thm}\label{thm:2}
$H(\fGCc_1,\delta + \co{m}{\cdot})$ is one dimensional, the class being represented by
$$
c:=\sum_{j> 0}
\frac{j}{(2j+1)!} \;
\begin{tikzpicture}[baseline=-.65ex,scale=.8]
 \node[int] (v) at (0,.5) {};
 \node[int] (w) at (0,-0.5) {};
 \draw (v) edge[very thick] node[right] {$\scriptstyle 2j+1$} (w);
\end{tikzpicture}.
$$
\end{thm}

The proof of this theorem will occupy the rest of this section. Let us however give a brief overview of the strategy, for the reader's convenience. There are four steps.
\begin{enumerate}
\item We introduce an auxiliary graph complex, the \emph{dotted complex}, which is larger than $\fGCc_1$ in that graphs may contain either even or odd ("dotted") edges. We show in Proposition \ref{prop:fGCcqiso} and Corollary \ref{cor:qi_dotted-extra} that this larger complex is still essentially quasi-isomorphic to $\fGCc_1$.
\item The (deformed) differential on the dotted complex contains a term which formally resembles the differential $\nabla$ (i.e., it adds one dotted edge in all possible ways) discussed in the previous section for the even-$d$-version of the graph complexes.
By a similar argument as that leading to the proof of Corollary \ref{cor:even} above we then reduce (in Proposition \ref{prop:qi_waved-dotted}) the dotted graph complex to a much simpler complex, which we call waved complex.
\item The waved complex is shown to be acyclic in Proposition \ref{prop:waved_acyclic}.
\item From this result we then finally deduce Theorem \ref{thm:2} in section \ref{sec:thm2proof}.
\end{enumerate}

\begin{rem}\label{rem:MCMoyal}
More conceptually, the Maurer-Cartan element above arises as follows. The operad $\Gra_1$ acts the space of polynomial functions $\K[V]$ on any symplectic vector space $V$.
The Moyal product endows $\K[V]$ with an associative algebra structure, which can be seen to factor through $\Gra_1$, i.e., we obtain a map $\Ass\to \Gra_1$. By restriction to $\Lie\subset \Ass$ we obtain an operad map $\mu: \Lie \to \Gra_1$.
The Maurer-Cartan element $m$ above can be seen as the Maurer-Cartan element in the deformation complex (cf. \cite{grt}) $\Def(\hoLie_1\to \Gra_1)=:\fGC_1$ corresponding to $\mu$.
The twisted graph complex may then be understood as a stable version of the Chevalley complex of the Moyal algebra.
Choosing Darboux coordinates on the symplectic vector space $V$ we may identify $V$ with $\K^{n+n}$ with the standard symplectic structure and the Moyal algebra with the differential operators on $\K^n$. Hence the twisted graph complex is a stable version of the Chevalley complex of the Lie algebra of differential operators, and Theorem \ref{thm:2} states that its cohomology is essentially trivial.
\end{rem} 

\subsection{Dotted complex} 
Let $\Gra^{\lin\dotte}_1$ be an operad defined analogously to $\Gra_1$ (cf. section \ref{sec:grcomplexes}), but with an additional type of edges
\begin{tikzpicture}[baseline=-.65ex,scale=.5]
 \node[int] (a) at (0,0) {};
 \node[int] (c) at (1,0) {};
 \draw (a) edge[dotted] (c);
\end{tikzpicture}
which we call dotted edges, and which we assign degree $1$.
It is understood that an ordering of the dotted edges is fixed, and two graphs with orderings differing by some permutation are identified up to the sign of the permutation. In particular, graphs with multiple dotted edges are zero.
Dotted tadpoles are not allowed by definition.
We endow $\Gra^{\lin\dotte}_1$ with a differential 
 $$
 \delta_1=\sum_{e\text{ multiple edge}}\delta_1^e
 $$
 where the sum ranges over all pairs of vertices with more than one (standard, not dotted) edge between them, and
 $$
\delta_1^e\;
\begin{tikzpicture}[baseline=-.65ex,scale=.8]
 \node[ext] (b) at (0,0) {};
 \node[ext] (a) at (1,0) {};
 \draw (a) edge[very thick] node[above] {$\scriptstyle n$} (b);
 \draw (b) edge (-.3,0.3);
 \draw (b) edge (-.4,0.1);
 \draw (b) edge (-.3,-0.3);
 \draw (b) edge (-.4,-0.1);
 \draw (a) edge (1.3,0.3);
 \draw (a) edge (1.4,0.1);
 \draw (a) edge (1.3,-0.3);
 \draw (a) edge (1.4,-0.1);
\end{tikzpicture}
=-n(n-1)\;
\begin{tikzpicture}[baseline=-.65ex,scale=.8]
 \node[ext] (b) at (0,0) {};
 \node[ext] (a) at (1,0) {};
 \draw (a) edge[very thick,bend right=20] node[above] {$\scriptstyle n-2$} (b);
 \draw (a) edge[dotted,bend left=20] (b);
 \draw (b) edge (-.3,0.3);
 \draw (b) edge (-.4,0.1);
 \draw (b) edge (-.3,-0.3);
 \draw (b) edge (-.4,-0.1);
 \draw (a) edge (1.3,0.3);
 \draw (a) edge (1.4,0.1);
 \draw (a) edge (1.3,-0.3);
 \draw (a) edge (1.4,-0.1);
\end{tikzpicture},
$$
where, to fix the sign, the standard edges are considered all to be oriented into the same direction and the dotted edge becomes the first in the ordering.
\begin{lemma}
The differential $\delta_1$ on $\Gra^{\lin\dotte}_1$ squares to zero and is a derivation with respect to the operad structure.
\end{lemma}
\begin{proof}
We first note that operators $\delta_1^e$ and $\delta_1^{e'}$ operating on different edges $e$, $e'$ anticommute and, trivially $(\delta_1^e)^2=0$. Hence $\delta_1^2=0$. Next let us verify that the differential is compatible with the operadic insertions. Considering two graphs $\Gamma$, $\Gamma'$, we want to check that
\begin{equation}\label{equ:compatibility}
\delta_1(\Gamma \circ_1\Gamma') = \Gamma \circ_1\Gamma' \pm \Gamma \circ_1\delta_1\Gamma',
\end{equation}
where $\circ_1$ denotes insertion into the first vertex. The edges in $\Gamma \circ_1\Gamma'$ are of three sorts: (i) edges coming from edges in $\Gamma'$, (ii) edges coming from edges in $\Gamma$ not incident to vertex 1 and (iii) edges coming from edges in $\Gamma$ incident to vertex 1, that have been reconnected to some vertex in $\Gamma'$.
On the other hand, the contributions to the differentials on the right-hand side of \eqref{equ:compatibility} stem (a) from edges of $\Gamma'$ (second term), (b) from edges of $\Gamma$ not incident to vertex 1 and (c) from edges of $\Gamma$ incident to vertex one.
Clearly, the pieces of the differential $\delta_1^e$ for $e$ an edge of types (i) and (ii) on the left-hand side of \eqref{equ:compatibility} cancel the terms for edges (a) and (b) on the right hand side. The only slightly non-trivial fact is that the contributions (iii) on the left cancel the contributions (c) on the right. To check this, it is sufficient to consider a graph $\Gamma'$ without edges, and an $n$-fold edge $e$ in $\Gamma$ connecting to vertex 1.
The cancellation is then clear from the following (schematic) commutative diagram.
\[
\begin{tikzcd}
\Gamma\otimes \Gamma'  = 
\tikz{\draw node[ext] (v) {1} edge[very thick] node[left] {$\scriptstyle n$} +(0,.75) +(0,.75) node[ext] (w){}; } 
\otimes 
\tikz{\draw node[ext] {$\scriptstyle 1$} ++(.5,0) node {$\cdots$} ++(.5,0) node[ext] {$\scriptstyle k$};}
\arrow{r}{\circ_1}
\arrow{d}{\delta_1\otimes \mathit{id}}
& 
\displaystyle\sum_{\substack{j_1,\dots,j_k \\ j_1+\dots+j_k=n}} \frac{n!}{j_1!\cdots j_k!}
\tikz{\draw node[ext] (v) {$\scriptstyle 1$} ++(.5,0) node {$\cdots$} ++(.5,0) node[ext] (w) {$\scriptstyle k$};
\node[ext](x) at (.5,.75){};
\draw[very thick] (x) edge node[ left] {$\scriptstyle j_1$} (v) edge node[right] {$\scriptstyle j_k$} (w);}
\arrow{r}{\delta_1}
&
-\displaystyle\sum_{r=1}^k \displaystyle\sum_{\substack{j_1,\dots,j_k \\ j_1+\dots+j_k=n}} \frac{n! j_r(j_r-1)}{j_1!\cdots j_k!} 
\tikz{\draw node[ext] (v) {$\scriptstyle 1$} ++(.5,0) node {$\cdots$} ++(.5,0) node[ext] (r) {$\scriptstyle r$}++(.5,0) node {$\cdots$} ++(.5,0) node[ext] (w) {$\scriptstyle k$};
\node[ext](x) at (1,1){};
\draw[very thick] (x) edge node[ left] {$\scriptstyle j_1$} (v) edge node[above] {$\scriptstyle j_k$} (w)
edge node[right]{$\scriptstyle j_r-2$} (r) edge[dotted,bend right] (r);}
\arrow{d}{=}
 \\
 -n(n-1) 
 \tikz{\draw node[ext] (v) {1} edge[very thick] node[right] {$\scriptstyle n-2$} +(0,.75) edge[bend left, dotted] +(0,.75) +(0,.75) node[ext] (w){}; } 
\otimes 
\tikz{\draw node[ext] {$\scriptstyle 1$} ++(.5,0) node {$\cdots$} ++(.5,0) node[ext] {$\scriptstyle k$};}
\arrow{rr}{\circ_1}
  & 
  &
 -\displaystyle\sum_{r=1}^k \displaystyle\sum_{\substack{j_1,\dots,j_k \\ j_1+\dots+j_k=n-2}} \frac{n(n-1)(n-2)!}{j_1!\cdots j_k!}
 \tikz{\draw node[ext] (v) {$\scriptstyle 1$} ++(.5,0) node {$\cdots$} ++(.5,0) node[ext] (r) {$\scriptstyle r$}++(.5,0) node {$\cdots$} ++(.5,0) node[ext] (w) {$\scriptstyle k$};
\node[ext](x) at (1,1){};
\draw[very thick] (x) edge node[ left] {$\scriptstyle j_1$} (v) edge node[above] {$\scriptstyle j_k$} (w)
edge node[right]{$\scriptstyle j_r$} (r) edge[dotted,bend right] (r);}
  \\
\end{tikzcd}
\]
\end{proof}

There is a natural operad map $\Lie\to \Gra^{\lin\dotte}_1$ sending the generator to the graph with two vertices and one solid edge.
Now the total space (of coinvariants) of the operad $\Gra^{\lin\dotte}_1$ forms a dg Lie algebra $\fGC^{\lin\dotte}_1$, in the same manner that we build $\fGC_1$ from $\Gra_1$ in section \ref{sec:grcomplexes}. More concretely, the differential on $\fGC^{\lin\dotte}_1$ has the form $\tilde\delta=\delta + \delta_1$ where $\delta=[
\begin{tikzpicture}[baseline=-.65ex, scale=.5]
 \draw (0,0) node[int] {} -- (1,0) node[int] {};
\end{tikzpicture}
,\cdot]$
 (cf. \eqref{equ:deltadef}) and $\delta_1$ is defined as above.

There is a natural projection of dg operads $\Gra^{\lin\dotte}_1\to \Gra_1$ sending a graph without dotted edges to itself and a graph with at least one dotted edge to $0$. This projection induces a projection of dg Lie algebras between the graph complexes $f\colon\fGC^{\lin\dotte}_1\rightarrow \fGC_1$.
Indeed we are interested in the connected part of the complexes, $\fGCc_1\subset\fGC_1$ and $\fGCc^{\lin\dotte}_1\subset\fGC^{\lin\dotte}_1$, spanned by only the connected graphs. The restriction $f\colon\fGCc^{\lin\dotte}_1\rightarrow \fGCc_1$ is well defined.

The subcomplex $\fGCc_1^\nomul\subset\fGCc_1$ generated by all graphs without multiple edges is also a subcomplex of $\fGCc^{\lin\dotte}_1$. We have the following proposition.

\begin{prop}\label{prop:fGCcqiso}
The inclusion $\fGCc_1^\nomul\rightarrow\fGCc^{\lin\dotte}_1$ is a quasi-isomorphism.
\end{prop}
\begin{proof}
On both complexes we set up a spectral sequence such that the first differential does not change the number of vertices, and we get the inclusion $(\fGCc_1^\nomul,0)\rightarrow(\fGCc^{\lin\dotte}_1,\delta_1)$. The second complex is a direct sum $(\fGCc^{\lin\dotte}_1,\delta_1)=(\fGCc_1^\nomul,0)\oplus (C,\delta_1)$ where $(C,\delta_1)\subset(\fGCc^{\lin\dotte}_1,\delta_1)$ is a subcomplex generated by graphs which have at least one dotted edge or double edge. Using a homotopy $h\colon C\rightarrow C$ which transforms the dotted edge to a double edge it can be seen that $(C,\delta_1)$ is acyclic, thus concluding the proof.
\end{proof}

We formally denote the ``dotted tadpole'' graph by
\begin{equation}\label{equ:pdef}
p:=
\begin{tikzpicture}[every loop/.style={}]
 \node[int] (a) at (0,0) {};
 \draw (a) edge[dotted,loop] (a);
\end{tikzpicture}
\end{equation}
and let $\fGC^{\lin\dotte\;p}_1:=\fGC^{\lin\dotte}_1\oplus\K p$. The connected version is $\fGCc^{\lin\dotte\;p}_1:=\fGCc^{\lin\dotte}_1\oplus\K p$. The Lie bracket on $\fGC^{\lin\dotte}_1$ naturally extends to $\fGC^{\lin\dotte\;p}_1$, the Lie bracket with $p$ being the operation of adding one dotted edge, in all possible ways.
The projection $f\colon \fGC^{\lin\dotte}_1\rightarrow\fGC_1$ extends to a function $f\colon \fGC^{\lin\dotte\;p}_1\rightarrow\fGC_1$ by setting $f(p)=0$, and it is a map of complexes.

\begin{cor}
\label{cor:dotted}
The projection $f\colon \fGCc^{\lin\dotte\;p}_1\rightarrow\fGCc_1$ is a quasi-isomorphism.
\end{cor}
\begin{proof}
Let $\Theta:=
\begin{tikzpicture}[baseline=-.65ex, scale=.5]
  \node[int] (v) at (0,.5){};
  \node[int] (w) at (0,-.5){};
  \draw (v) edge (w) edge[bend left=35] (w) edge[bend right=35] (w);
 \end{tikzpicture}$
and $\zeta:=
\begin{tikzpicture}[baseline=-.65ex,scale=.5]
 \node[int] (a) at (0,.5) {};
 \node[int] (b) at (0,-.5) {};
 \draw (a) edge[dotted,bend right=20] (b);
 \draw (a) edge[bend left=20] (b);
\end{tikzpicture}$.
Let $\fGCc_1^{\nomul\;\Theta}:=\fGCc_1^\nomul\oplus\K\Theta$ be the subcomplex of $(\fGCc_1,\delta)$ spanned by $\fGCc_1^\nomul$ and the graph $\Theta$, and let similarly $\fGCc_1^{\nomul\;\Theta p \zeta}:=\fGCc_1^\nomul\oplus\K\Theta\oplus\K p\oplus\K\zeta$ be the subcomplex of $(\fGCc^{\lin\dotte\;p}_1,\tilde\delta)$ spanned by $\fGCc_1^\nomul$ and the graphs $\Theta$, $p$ and $\zeta$.
It follows from Proposition \ref{prop:fGCcqiso} that the inclusion $\fGCc_1^{\nomul\;\Theta p \zeta}\rightarrow\fGCc^{\lin\dotte\;p}_1$ is a quasi-isomorphism. Clearly, the projection $\fGCc_1^{\nomul\;\Theta p \zeta}\rightarrow\fGCc_1^{\nomul\;\Theta}$ is also a quasi-isomorphism. \cite[Theorem 2]{eulerchar} implies that the inclusion $\fGCc_1^{\nomul\;\Theta}\rightarrow\fGCc_1$ is also a quasi-isomorphism. Composing all these maps on the level of homology leads to the result.
\end{proof}

\subsection{Extra differential on the dotted complex}
We consider the following degree 1 element of $\fGC^{\lin\dotte\;p}_1$:
\begin{equation}\label{equ:tildemdef}
\tilde m \; := \; m \; + p
= \;
\sum_{j\geq 1}
\frac{1}{(2j+1)!} \;
\begin{tikzpicture}[baseline=-.65ex,scale=.8]
 \node[int] (v) at (0,.5) {};
 \node[int] (w) at (0,-0.5) {};
 \draw (v) edge[very thick] node[right] {$\scriptstyle 2j+1$} (w);
\end{tikzpicture}
+
\begin{tikzpicture}[every loop/.style={}]
 \node[int] (a) at (0,0) {};
 \draw (a) edge[dotted,loop] (a);
\end{tikzpicture}.
\end{equation}

\begin{lemma}
The element $\tilde m$ is a Maurer-Cartan element in $\fGC^{\lin\dotte\;p}_1$, i.~e., 
$$\tilde\delta\tilde m + \frac{1}{2}\co{\tilde m}{\tilde m}=0.$$
\end{lemma}
\begin{proof}
It holds that
\begin{align*}
\tilde\delta\tilde m + \frac{1}{2}\co{\tilde m}{\tilde m} & =
\delta  m + \frac{1}{2}\co{m}{m} + \delta
p
+\delta_1m + \delta_1p + \co{m}{p}
+\frac{1}{2}
\co{p}{p}\\
&=
\zeta+
\delta_1 m + \co{m}{p} = 
\delta_1 m + \co{m'}{p}.
\end{align*}
But now
\[
\delta_1 m
=
\sum_{j\geq 1}
\frac{-2j(2j+1)}{(2j+1)!} \;
\begin{tikzpicture}[baseline=-.65ex,scale=.8]
 \node[int] (v) at (0,.5) {};
 \node[int] (w) at (0,-0.5) {};
 \draw (v) edge[very thick] node[right] {$\scriptstyle 2j-1$} (w)
       (v) edge[dotted, bend right] (w);
\end{tikzpicture}
=
-
 \sum_{j\geq 0}
\frac{1}{(2j+1)!} \;
\begin{tikzpicture}[baseline=-.65ex,scale=.8]
 \node[int] (v) at (0,.5) {};
 \node[int] (w) at (0,-0.5) {};
 \draw (v) edge[very thick] node[right] {$\scriptstyle 2j+1$} (w)
       (v) edge[dotted, bend right] (w);
\end{tikzpicture}
=
-\co{m'}{p}
\]
and hence the Lemma is shown.
\end{proof}

Because of the lemma we may compute the cohomology of $\fGC^{\lin\dotte\;p}_1$ with respect to the differential $\tilde\delta + \co{\tilde m}{\cdot} = \delta_1 + \co{m'+p}{\cdot}$. We have $f(\tilde m)=m$, so $f$ is a map of complexes $(\fGCc^{\lin\dotte\;p}_1,\tilde\delta + \co{\tilde m}{\cdot})\rightarrow(\fGCc_1,\delta+\co{m}{\cdot})$. We need a general statement.

\begin{prop}\label{prop:9}
Let $C$ and $D$ be differential graded Lie algebras equipped with descending complete filtrations $C=\mF^0C\supset \mF^1C\supset\cdots$, $D=\mF^0D\supset \mF^1D\supset\cdots$. We assume that these filtrations are compatible with the Lie structure in the sense that $[\mF^pC, \mF^qC]\subset \mF^{p+q}C$ etc., and that the spaces $\mF^pC/\mF^{p+1}C$ and $\mF^pD/\mF^{p+1}D$ are finte dimensional in each (cohomological) degree. Let $m\in \mF^1C$ be Maurer-Cartan element.
If $f\colon C\rightarrow D$ be a morphism of differential graded Lie algebras respecting the filtrations, and inducing a quasi-isomorphism on the associated graded complexes.
Then $f(m)\in \mF^1D$ is Maurer-Cartan element in $D$ and $f$ induces a quasi-isomorphism of the twisted dg Lie algebras $(C,\delta + [m,\cdot])\to (D,\delta + [f(m),\cdot])$.
\end{prop}
\begin{proof}
Consider the spectral sequences on $C$ and $D$ arising from the filtrations given. Since $f$ is a quasi-isomorphism on that page by assumption, it will be isomorphism on the second page. Remaining differentials on $D$ are defined using $f$ from those on $C$, so they commute with $f$ and all the remaining pages are the same. Using the finite-dimensionality constraint together with Proposition \ref{prop:App2} we see that both spectral sequences weakly converge to the cohomologies of complexes, hence the result.
\end{proof}

\begin{cor}
\label{cor:qi_dotted-extra}
The projection $f\colon (\fGCc^{\lin\dotte\;p}_1,\tilde\delta+\co{\tilde m}{\cdot})\rightarrow(\fGCc_1,\delta + \co{m}{\cdot})$ is a quasi-isomorphism.
\end{cor}
\begin{proof}
We define the "Lie degree" of a graph to be the number of edges plus twice the number of dotted edges minus the number of vertices, and the filtration such that the $p$-th subspace is spanned by graphs of Lie degree $\geq p$. The finite dimensionality condition in Proposition \ref{prop:9} is easily checked, and using Corollary \ref{cor:dotted}, the Proposition implies the result.
\end{proof}

\subsection{Waved complex}

Let $\fGC^\snak_{\text{full}}$ be a graph complex generated by graphs of vertices of degree 1 and directed edges
\begin{tikzpicture}[baseline=-.65ex,scale=.5]
 \node[int] (a) at (0,0) {};
 \node[int] (c) at (1,0) {};
 \draw (a) edge[snakeit,->] (c);
\end{tikzpicture}
of degree $0$. The direction of the edge can not be reversed, and every pair of vertices is connected with exactly one edge, of some orientation.
In other words, disregrading the edge orientations all the graphs are full graphs.


One differential on $\fGC^\snak_{\text{full}}$ is $d_1(\Gamma)=\sum_{x\in V(\Gamma)}d_x(\Gamma)$ where $d_x$ adds another vertex $v+1$, connects the vertices $x$ and $v+1$
\begin{tikzpicture}[baseline=-.65ex,scale=.5]
 \node[int] (a) at (0,0) {};
 \node[below] at (a) {$\scriptstyle x$};
 \node[int] (c) at (1,0) {};
 \node[below] at (c) {$\scriptstyle v+1$};
 \draw (a) edge[snakeit,->] (c);
\end{tikzpicture},
and creates for every other vertex $y$ of $\Gamma$ an edge from $y$ to $v+1$ or from $v+1$ to $y$, depending on whether the edge between $y$ and $x$ was oriented from $y$ to $x$ or vice versa.
Another differential is $d:=d_1-d_\text{in}+d_\text{out}$ where
$d_\text{in}$ adds the vertex $v+1$ and connects every other vertex to it with an edge directed towards it, and $d_\text{out}$ does the same with the opposite direction. It can easily be seen that both of them are differentials, i.e.\ that $d_1^2=d^2=0$.

%
%

\begin{prop}
\label{prop:waved_acyclic}
The complex $(\fGC^\snak_{\text{full}},d)$ is acyclic.
\end{prop}
\begin{proof}
We look for a homotopy $h\colon \fGC^\snak_{\text{full}}\rightarrow\fGC^\snak_{\text{full}}$ of degree $-1$ such that $hd+dh=-\Id$.

Let \emph{the maximum} vertex in the graph $\Gamma$ be the vertex for which all adjacent edges head towards it, if such a vertex exists. Let the map $h$ delete the maximum vertex of a graph $\Gamma$ and its adjacent edges if such a vertex exists, and set $h(\Gamma)=0$ otherwise. The maps $d_1$ and $d_\text{out}$ can not create or remove the maximum vertex, so if there is no such vertex in $\Gamma$ then $(hd+dh)(\Gamma)=-hd_\text{in}(\Gamma)=-\Gamma$. On the other hand, If $x$ is the maximum vertex, then $(hd+dh)(\Gamma)=(hd_x-hd_\text{in}-d_\text{in}h)(\Gamma)=-\Gamma$.
\end{proof}

We construct a map $g\colon\fGC^\snak_{\text{full}}\rightarrow\fGC^{\lin\dotte}_1$ which maps a graph to the graph with the same vertices, and edges
\begin{tikzpicture}[baseline=-.65ex,scale=.5]
 \node[int] (a) at (0,0) {};
 \node[below] at (a) {$\scriptstyle x$};
 \node[int] (c) at (1,0) {};
 \node[below] at (c) {$\scriptstyle y$};
 \draw (a) edge[snakeit,->] (c);
\end{tikzpicture}
are replaced with
$$
\sum_{j\geq 0} \frac{1}{j!} \;
\begin{tikzpicture}[baseline=-.65ex,scale=.8]
 \node[int] (v) at (-.5,0) {};
 \node[below] at (v) {$\scriptstyle x$};
 \node[int] (w) at (.5,0) {};
 \node[below] at (w) {$\scriptstyle y$};
 \draw (v) edge[very thick] node[above] {$\scriptstyle j$} (w);
\end{tikzpicture}
$$
where the direction of all edges is the same as the direction of the former waved edge (from $x$ to $y$).

\begin{prop}
The map $g:(\fGC^\snak_{\text{full}},d)\rightarrow(\fGC^{\lin\dotte}_1,\tilde\delta+\co{\tilde m}{\cdot})$ is a map of complexes, i.e.\ $g(d\Gamma)=\tilde\delta g(\Gamma)+\co{\tilde m}{g(\Gamma)}$.
\end{prop}
\begin{proof}
We can write $\tilde\delta g(\Gamma)+\co{\tilde m}{g(\Gamma)}=\delta g(\Gamma)+\delta_1 g(\Gamma)+\co{m}{g(\Gamma)}+\co{p}{g(\Gamma)}$. The part $\delta_1+\co{p}{\cdot}$ does not change the number of vertices, and acts on every pair of them separately. It is easy to see that it sends the image of $g$ to $0$. Therefore it is enough to prove
$$
g(d\Gamma)=\delta g(\Gamma)+\co{m}{g(\Gamma)}=\co{m'}{g(\Gamma)}.
$$
Note that
$$
m'=\sum_{j\geq 0}\frac{1}{(2j+1)!} \;
\begin{tikzpicture}[baseline=-.65ex,scale=.8]
 \node[int] (v) at (0,.5) {};
 \node[int] (w) at (0,-0.5) {};
 \draw (v) edge[very thick] node[right] {$\scriptstyle 2j+1$} (w);
\end{tikzpicture}=\sum_{j\geq 0}\frac{1}{j!} \;
\begin{tikzpicture}[baseline=-.65ex,scale=.8]
 \node[int] (v) at (0,.5) {};
 \node[int] (w) at (0,-0.5) {};
 \draw (v) edge[very thick] node[right] {$\scriptstyle j$} (w);
\end{tikzpicture}
$$
because the additional terms are $0$ by symmetry reasons. Therefore
\begin{align*}
\co{m'}{g(\Gamma)} &= m'\bullet g(\Gamma)-(-1)^{|\Gamma|}g(\Gamma)\bullet m'=\\
&= \sum_{j\geq 0}\frac{1}{j!}\;
\begin{tikzpicture}[baseline=-.65ex]
 \node[int] (v) at (-.4,0) {};
 \node[ext] (w) at (.4,0) {$g\Gamma$};
 \draw (v) edge[very thick] node[above] {$\scriptstyle j$} (w);
\end{tikzpicture}
- \sum_{j\geq 0}\frac{1}{j!}\;
\begin{tikzpicture}[baseline=-.65ex]
 \node[int] (v) at (.4,0) {};
 \node[ext] (w) at (-.4,0) {$g\Gamma$};
 \draw (v) edge[very thick] node[above] {$\scriptstyle j$} (w);
\end{tikzpicture}
-(-1)^{|\Gamma|} \sum_{x\in V(\Gamma)}\frac{1}{j!}\;
\begin{tikzpicture}[baseline=-.65ex]
 \node[ext] (v) at (.4,0) {
 \begin{tikzpicture}[baseline=-.65ex]
 \node[int] (v) at (0,.2) {};
 \node[int] (w) at (0,-0.2) {};
 \draw (v) edge[very thick] node[right] {$\scriptstyle j$} (w);
\end{tikzpicture} 
 };
 \node at (.5,-.5) {$\scriptstyle x$};
 \node (w) at (-.5,0) {$g\Gamma$};
 \draw (v) edge[double] (w);
\end{tikzpicture}=\\
&= g\left(
\begin{tikzpicture}[baseline=-.65ex]
 \node[int] (v) at (-.3,0) {};
 \node[ext] (w) at (.3,0) {$\Gamma$};
 \draw (v) edge[snakeit,->] (w);
\end{tikzpicture}\right)
- g\left(
\begin{tikzpicture}[baseline=-.65ex]
 \node[int] (v) at (.3,0) {};
 \node[ext] (w) at (-.3,0) {$\Gamma$};
 \draw (w) edge[snakeit,->] (v);
\end{tikzpicture}\right)
-(-1)^{|\Gamma|} g\left(\sum_{x\in V(\Gamma)}
\begin{tikzpicture}[baseline=-.65ex]
 \node[ext] (v) at (.4,0) {
 \begin{tikzpicture}[baseline=-.65ex]
 \node[int] (v) at (0,.2) {};
 \node[int] (w) at (0,-0.2) {};
 \draw (w) edge[snakeit,->] (v);
\end{tikzpicture} 
 };
 \node at (.4,-.5) {$\scriptstyle x$};
 \node (w) at (-.4,0) {$\Gamma$};
 \draw (v) edge[double] (w);
\end{tikzpicture}\right)=\\
&= g\left(d_\text{out}\Gamma\right)
- g\left(d_\text{in}\Gamma\right)
+g\left(\sum_{x\in V(\Gamma)}
d_x(\Gamma)\right)=g(d\Gamma).
\end{align*}
\end{proof}

\begin{prop}
\label{prop:qi_waved-dotted}
The map $g\colon(\fGC^\snak_{\text{full}},d)\rightarrow(\fGC^{\lin\dotte}_1,\tilde\delta+\co{\tilde m}{\cdot})$ is a quasi-isomorphism.
\end{prop}
\begin{proof}
On both complexes we set up the spectral sequence such that the first differential does not change the number of vertices. On the first page we have $(\fGC^\snak_{\text{full}},0)$ and $(\fGC^{\lin\dotte}_1,\delta_1+\co{p}{\cdot})$. The spectral sequence of the first complex clearly converges to its cohomology, and so does the second by virtue of Proposition \ref{prop:App1}, because for fixed cohomological degree (number of vertices plus number of dotted edges) the number of vertices is bounded.

Similar to the proof of Proposition \ref{prop:evenD} we can distinguish vertices for every particular number of them. The differential $\delta_1+\co{p}{\cdot}$ acts on every pair of vertices $(x,y)$ separately, so the complexes with the chosen number of vertices on the second page are tensor products of complexes for every pair $(x,y)$.

An easy investigation of the action of $\delta_1+\co{p}{\cdot}$ on the pair $(x,y)$ leads to the $2$ dimensional homology generated by $\displaystyle
\sum_{j\text{ even}} \frac{1}{j!} \;
\begin{tikzpicture}[baseline=-.65ex]
 \node[int] (v) at (-.3,0) {};
 \node[below] at (v) {$\scriptstyle x$};
 \node[int] (w) at (.3,0) {};
 \node[below] at (w) {$\scriptstyle y$};
 \draw (v) edge[very thick] node[above] {$\scriptstyle j$} (w);
\end{tikzpicture}
$ and $\displaystyle
\sum_{j\text{ odd}} \frac{1}{j!} \;
\begin{tikzpicture}[baseline=-.65ex]
 \node[int] (v) at (-.3,0) {};
 \node[below] at (v) {$\scriptstyle x$};
 \node[int] (w) at (.3,0) {};
 \node[below] at (w) {$\scriptstyle y$};
 \draw (v) edge[very thick] node[above] {$\scriptstyle j$} (w);
\end{tikzpicture}
$. A different basis is $\left\{g\left(
\begin{tikzpicture}[baseline=-.65ex,scale=.5]
 \node[int] (a) at (0,0) {};
 \node[below] at (a) {$\scriptstyle x$};
 \node[int] (c) at (1,0) {};
 \node[below] at (c) {$\scriptstyle y$};
 \draw (a) edge[snakeit,->] (c);
\end{tikzpicture}
\right), g\left(
\begin{tikzpicture}[baseline=-.65ex,scale=.5]
 \node[int] (a) at (0,0) {};
 \node[below] at (a) {$\scriptstyle x$};
 \node[int] (c) at (1,0) {};
 \node[below] at (c) {$\scriptstyle y$};
 \draw (c) edge[snakeit,->] (a);
\end{tikzpicture}
\right)\right\}$. Now it is easy to see that $g$ on the second page is actually an isomorphism.
\end{proof}

\subsection{End of proof}\label{sec:thm2proof}
Now we can obtain the cohomology of the complex $\fGC_1$ with respect to the differential $\delta + \co{m}{\cdot}$.

\begin{proof}[Proof of Theorem \ref{thm:2}]
Propositions \ref{prop:waved_acyclic} and \ref{prop:qi_waved-dotted} imply that $(\fGC^{\lin\dotte}_1,\tilde\delta+\co{\tilde m}{\cdot})$ is acyclic.

We claim that the connected part $(\fGCc^{\lin\dotte}_1,\tilde\delta+\co{\tilde m}{\cdot})$ is also acyclic. Suppose the opposite, and let the first class in cohomology appear in degree $t$, i.e., the combined number of vertices and dotted edges is $t$. Now $t$ can not be $1$, so $t\geq 2$. We set up a filtration on $(\fGC^{\lin\dotte}_1,\tilde\delta+\co{\tilde m}{\cdot})$ according to the number of connected components
 and a spectral sequence such that the first differential does not connect them.
 Note that this filtration is automatically bounded: the number of connected components is certainly bounded below, and it is bounded above since for a fixed cohomological degree one can have only a finite number of connected components owed to the fact that vertices carry degree one and edges carry non-negative degrees. Hence the spectral sequence in particular converges to the cohomology, i.e., to 0.
 
 Since the differential on the associated graded respects the numbers of connected components by construction, we will find the symmetric product of $H(\fGCc^{\lin\dotte}_1,\tilde\delta+\co{\tilde m}{\cdot})$ on the $E^1$-page of the spectral sequence. The subsequent differentials in the spectral sequence will (i) reduce the numbers of connected components and (ii) increase the degree by 1. But by our (contrapositive) assumption we have a lowest degree connected cohomology class in degree $t\geq 2$. 
 But then the higher symmetric powers $S^k(H(\fGCc^{\lin\dotte}_1,\tilde\delta+\co{\tilde m}{\cdot}))$ will be concentrated 
 in degrees $kt\geq 2t$. Hence our degree $t$ connected class cannot be cancelled on further pages of the spectral sequence, contradicting the acyclicity of $(\fGC^{\lin\dotte}_1,\tilde\delta+\co{\tilde m}{\cdot})$.
 Hence we conclude that $H(\fGCc^{\lin\dotte}_1,\tilde\delta+\co{\tilde m}{\cdot})=0$.
 
Adding $p$ to $(\fGCc^{\lin\dotte}_1,\tilde\delta+\co{\tilde m}{\cdot})$ (see \eqref{equ:pdef}) creates one class in cohomology. In fact, an explicit generating cocycle may be constructed as follows. In general, if $\tilde m$ is a Maurer-Cartan element in a differential graded Lie algebra $\mathfrak{g}$ and if $D$ is some derivation of $\mathfrak{g}$, then $D\tilde m$ is a cocycle in the twisted dg Lie algebra $\mathfrak g^{\tilde m}$. Furthermore, if $\mathfrak g$ carries an additional grading, then the grading generator (multiplying a homogeneous element of degree $\alpha$ by $\alpha$) is a derivation.
In our case, the dg Lie algebra $(\fGCc^{\lin\dotte,p}_1,\tilde\delta)$ carries a grading by the loop order, plus the number of dotted edges. Hence, applying the grading generator to the Maurer-Cartan element $\tilde m$ (cf. \eqref{equ:tildemdef}) yields the desired cocycle in $(\fGCc^{\lin\dotte,p}_1,\tilde\delta+\co{\tilde m}{\cdot})$
\[
\tilde c :=
\sum_{j\geq 1}
\frac{2j}{(2j+1)!} \;
\begin{tikzpicture}[baseline=-.65ex,scale=.8]
 \node[int] (v) at (0,.5) {};
 \node[int] (w) at (0,-0.5) {};
 \draw (v) edge[very thick] node[right] {$\scriptstyle 2j+1$} (w);
\end{tikzpicture}
+
2
\begin{tikzpicture}[every loop/.style={}]
 \node[int] (a) at (0,0) {};
 \draw (a) edge[dotted,loop] (a);
\end{tikzpicture}.
\]
Corollary \ref{cor:qi_dotted-extra} now implies that $H(\fGCc_1,\delta + \co{m}{\cdot})$ is one dimensional with the class being represented by 
\[
f(\tilde c)=
\sum_{j> 0}
\frac{2j}{(2j+1)!} \;
\begin{tikzpicture}[baseline=-.65ex,scale=.8]
 \node[int] (v) at (0,.5) {};
 \node[int] (w) at (0,-0.5) {};
 \draw (v) edge[very thick] node[right] {$\scriptstyle 2j+1$} (w);
\end{tikzpicture} =2c.
\]
\end{proof}

\begin{rem}
The proof of Theorem \ref{thm:2} we gave above is purely combinatorial. 
More conceptually, the result may be related to the rigidity of the associative operad, together with the presence of the Hodge filtration on this operad, as we explain in Appendix \ref{app:altproof}. 
\end{rem}

\subsection{A picture of the odd graph cohomology}
Similarly as in odd case, there is a Corollary.

\begin{cor}\label{cor:odd}
There is a spectral sequence converging to
\[
 H(\fGCc_1,\delta + \co{m}{\cdot})=\K c
\]
whose $E^1$ term is 
\[
 H(\fGCc_1,\delta).
\]
Furthermore, all differentials on odd pages are $0$.
\end{cor}
\begin{proof}
We set the spectral sequence on $(\fGCc_1,\delta + \co{m}{\cdot})$ such that the first differential does not change $b=e-v$. Clearly, on the $E^1$ page we have $H(\fGCc_1,\delta)$.

The complex $(\fGCc_1,\delta + \co{m}{\cdot})$ is in fact the direct sum of two complexes, namely its subcomplexes corresponding to even and odd $b$. Therefore, there can not be a non-zero differential changing the parity of $b$ on any page of spectral sequence, and hence all differentials on odd pages are $0$.

The subspace of $\fGCc_1$ corresponding to a fixed number of vertices and a fixed $b$ is finite-dimensional, so Proposition \ref{prop:App2} implies that spectral sequence converges to the cohomology of the complex.
\end{proof}

Table \ref{tbl:oddcanceling} represents the second page of the spectral sequence from Corollary \ref{cor:odd}, i.e.\ $H(\fGC_1,\delta)$, where the column number represents the number of vertices $v$ and the row number represents the Betti number minus one $b=e-v$. The numbers in the table represent the dimension of the respective subspace of $H(\fGC_1,\delta)$. They were partially taken from \cite{barnatanmk}, and partially calculated by a computer program of the second author.

\begin{table}
\begin{tikzpicture}
\matrix (mag) [matrix of nodes,ampersand replacement=\&]
{
{}\& 1 \& 2 \& 3 \& 4 \& 5 \& 6 \& 7 \& 8 \& 9 \&10 \&11 \&12 \&13 \&14 \&15 \&16 \&17 \&18 \&19 \&20 \&21 \&22 \\
-1\& 0 \& 0 \& 0 \& 0 \& 0 \& 0 \& 0 \& 0 \& 0 \& 0 \& 0 \& 0 \& 0 \& 0 \& 0 \& 0 \& 0 \& 0 \& 0 \& 0 \& 0 \& 0 \\
0 \& 0 \& 0 \& 1 \& 0 \& 0 \& 0 \& 1 \& 0 \& 0 \& 0 \& 1 \& 0 \& 0 \& 0 \& 1 \& 0 \& 0 \& 0 \& 1 \& 0 \& 0 \& 0 \\
1 \& 0 \& 1 \& 0 \& 0 \& 0 \& 0 \& 0 \& 0 \& 0 \& 0 \& 0 \& 0 \& 0 \& 0 \& 0 \& 0 \& 0 \& 0 \& 0 \& 0 \& 0 \& 0 \\
2 \& 0 \& 0 \& 0 \& 1 \& 0 \& 0 \& 0 \& 0 \& 0 \& 0 \& 0 \& 0 \& 0 \& 0 \& 0 \& 0 \& 0 \& 0 \& 0 \& 0 \& 0 \& 0 \\
3 \& 0 \& 0 \& 0 \& 0 \& 0 \& 1 \& 0 \& 0 \& 0 \& 0 \& 0 \& 0 \& 0 \& 0 \& 0 \& 0 \& 0 \& 0 \& 0 \& 0 \& 0 \& 0 \\
4 \& 0 \& 0 \& 0 \& 0 \& 0 \& 0 \& 0 \& 2 \& 0 \& 0 \& 0 \& 0 \& 0 \& 0 \& 0 \& 0 \& 0 \& 0 \& 0 \& 0 \& 0 \& 0 \\
5 \& 0 \& 0 \& 0 \& 0 \& 0 \& 0 \& 1 \& 0 \& 0 \& 2 \& 0 \& 0 \& 0 \& 0 \& 0 \& 0 \& 0 \& 0 \& 0 \& 0 \& 0 \& 0 \\
6 \& 0 \& 0 \& 0 \& 0 \& 0 \& 0 \& 0 \& 0 \& 1 \& 0 \& 0 \& 3 \& 0 \& 0 \& 0 \& 0 \& 0 \& 0 \& 0 \& 0 \& 0 \& 0 \\
7 \& 0 \& 0 \& 0 \& 0 \& 0 \& 0 \& 0 \& 0 \& 0 \& 0 \& 2 \& 0 \& 0 \& 4 \& 0 \& 0 \& 0 \& 0 \& 0 \& 0 \& 0 \& 0 \\
8 \& 0 \& 0 \& 0 \& 0 \& 0 \& 0 \& 0 \& 0 \& 0 \& ? \& 0 \& ? \& ? \& ? \& ? \& 5 \& 0 \& 0 \& 0 \& 0 \& 0 \& 0 \\
9 \& 0 \& 0 \& 0 \& 0 \& 0 \& 0 \& 0 \& 0 \& 0 \& 0 \& ? \& ? \& ? \& ? \& ? \& ? \& ? \& 6 \& 0 \& 0 \& 0 \& 0 \\
10\& 0 \& 0 \& 0 \& 0 \& 0 \& 0 \& 0 \& 0 \& 0 \& 0 \& 0 \& ? \& ? \& ? \& ? \& ? \& ? \& ? \& ? \& 8 \& 0 \& 0 \\
11\& 0 \& 0 \& 0 \& 0 \& 0 \& 0 \& 0 \& 0 \& 0 \& 0 \& 0 \& 0 \& ? \& ? \& ? \& ? \& ? \& ? \& ? \& ? \& ? \& 9 \\
};
\draw (mag-2-1.north west) -- (mag-1-23.south east);
\draw (mag-1-2.north west) -- (mag-14-1.south east);
\draw[-latex, thick] (mag-3-4) edge (mag-5-5);
\draw[-latex, thick] (mag-3-8) edge (mag-7-9);
\draw[-latex, thick] (mag-3-12) edge node[above]{$?$} (mag-9-13);
\draw[-latex, thick] (mag-3-16) edge node[above]{$?$} (mag-11-17);
\draw[-latex, thick] (mag-3-20) edge node[above]{$?$} (mag-13-21);
\draw[-latex, thick] (mag-6-7) edge (mag-8-8);
\draw[-latex, thick] (mag-7-9) edge (mag-9-10);
\draw[-latex, thick] (mag-8-11) edge node[above]{$?$} (mag-10-12);
\end{tikzpicture}
\caption{\label{tbl:oddcanceling} The table of cohomology of $\fGCc_1$. The arrows indicate some cancellations in the spectral sequence.}
\end{table}

The classes for $b=0$ are represented by the loop graphs $L_k$ (see Figure \ref{fig:loops}). Only the classes $L_{4j+3}$ are non-zero in $\fGCc_1$ by symmetry reasons.
Corollary \ref{cor:odd} implies that the classes in the graph cohomology occur in pairs, with one partner canceling the other on some page of the above spectral sequence. Since canceling can not occur on odd pages, all partners have the same parity of $b$. Some cancellations are represented by arrows in the table above. The arrows with a question mark are conjectural, since the computed data is not sufficient to rule out cancellations with classes in the unknown region. 

As a further application of Theorem \ref{thm:2}, we obtain a lower bound for the cohomology of the graph complex $\GC_3$ in degrees $\leq -3$. To this end we will denote by $H^{j}(\GC_3)_n$ the subspace of $H^{j}(\GC_3)$ of classes representable by graphs with fixed number of vertices $n$.

\begin{cor}\label{cor:odddimcor}
Let $h_n=\vdim H^{-3}(\GC_3)_n$. Then there are the following inequalities:
 \[
  \sum_{j\geq 1} \vdim H^{-4-j}(\GC_3)_{n+1} \geq h_n + \epsilon_n 
 \]
 where
 \[
  \epsilon_n=
  \begin{cases}
   1 & \text{if $n\equiv 0$ mod $4$ } \\
   0 & \text{otherwise}
  \end{cases}
 \]
\end{cor}

The non-loop cohomology classes found so far computationally live in cohomological degrees $-3$ and $-6$.
Currently, we do not know which classes in $\GC_3$ are the partners of the loop graphs $L_k$. However, we raise the following conjecture.
\begin{conjecture}\label{conj:wheelkills}
The classes killed by the loop graphs $L_k$ are linear combinations of trivalent graphs.
In other words, the corresponding classes live in degree $-3$ in $\GC_3$.
\end{conjecture}

Note that the classes $t,\omega_0, \omega_1,\dots$ introduced and shown to be non-trivial by P. Vogel \cite{vogel} live in the correct degrees to be the partners of the loops $L_3, L_7,\dots$. Indeed, it is easy to check that $L_3$ kills Vogel's $t$ (the Tetrahedron graph) in the spectral sequence. The partner of $L_7$ is computed in Appendix \ref{app:7loop}, and it differs from Vogel's $\omega_0$.


%
%
%

%
%


\appendix

 \section{Rigidity of the associative operad and Theorem \ref{thm:2}}\label{app:altproof}
 
The purpose of this Appendix is to sketch a different, conceptually more satisfying explanation of Theorem \ref{thm:2}.
We will refrain from giving a second complete proof of this Theorem, but only sketch the main ingredients.
We recall from \cite{grt} that one may express the graph cohomology as (a part of) the cohomology of the operadic deformation complex
\begin{equation}\label{equ:HDefPoiss}
H(\Def(\hoPoiss \to \Poiss))=S^+(\K[1]\oplus H(\fGCc_1)[-1])[2].
\end{equation}
It is well known that the associative operad $\Ass$ carries a filtration, the Hodge filtration, such that the associated graded operad is the Poisson operad. For a more detailed description of this filtration see for example \cite[sections 3.2 and 3.3]{TW}.
Using the Hodge filtration (both on the target and a variant on the source) we may induce a filtration on the deformation complex of the associative operad
\[
\Def(\Ass_\infty\to \Ass),
\]
such that the first convergent in the associated spectral sequence is \eqref{equ:HDefPoiss}. This can be seen as the conceptual origin of the spectral sequence of Theorem \ref{thm:main} in the odd case, provided the following lemma.
 
 \begin{lemma}\label{lem:assrigid}
 $\Def(\Ass_\infty\to \Ass)$ is acyclic. 
 \end{lemma}
 \begin{proof}
  Recall that 
 \[
  \Def(\Ass_\infty\to \Ass) = \prod_{N\geq 0} \Hom_{S_N}(\Ass^\vee(N),\Ass(N))[-1].
 \]
 Elements of the $N$-th factor can be understood as linear combinations of associative words in variables $a_1,a_2,\dots,a_N$, each appearing exactly once.
 The differential is a variant of the Hochschild differential:
 \[
  d = \sum_{i=0}^{N+1} (-1)^i d_i
 \]
 where $d_i$ ($1\leq i\leq N$) replaces $a_i$ by $a_ia_{i+1}$ and renumbers $a_{i+1}\to a_{i+2}$ etc., while $d_0$ ($d_{N+1}$) multiply the whole word from the left (right) by $a_1$ (by $a_{N+1}$) and renumber the other variables accordingly.
 
 To show that the above complex is acyclic we may proceed by induction on $N$. So we suppose that any cocycle of length $< N$ is exact, and want to show that any cocycle of length $N$ is exact.
 Define the homotopy 
 \begin{gather*}
  h_N:\Hom_{S_N}(\Ass^\vee(N),\Ass(N))[-1]\to \Hom_{S_{N-1}}(\Ass^\vee(N-1),\Ass(N-1))[-1] 
 \\
 h(W) := 
 \begin{cases}
  W' & \text{if $W=a_1W'$} \\
 0 & \text{if $W$ does not start with $a_1$}
 \end{cases}.
 \end{gather*}
 
 One checks that $h_{N+1}d+dh_{N}=1-\pi_N$, where $\pi_N$ is the projection onto those words starting with $a_1$.
 These words form a subcomplex quasi-isomorphic to the original complex, and hence by induction one shows the Lemma.
 \end{proof}
 
Let us be more precise and describe the relation to Theorem \ref{thm:2} using the above ingredients.
 For this, we will choose a particular model of the (target) associative operad, carrying an action of the (twisted) graph complex.  
As indicated in Remark \ref{rem:MCMoyal} we have a map 
 \[
  F: \Ass \to \Gra_1
 \]
 defined such 
 \[
  F(m_2) = 
 \sum_{j\geq 0}
 \frac 1 {j!}
 \begin{tikzpicture}[baseline=-.65ex]
  \node[ext] (v) at (0,0) {1};
  \node[ext] (w) at (1,0) {2};
  \node at (0.5,0.1) {$\vdots$};
 \draw (v) edge[bend left=50] (w) edge[bend right=50] (w);
 \node at (.5,-.5) {$j$ edges};
 \end{tikzpicture}
 \]
 \begin{rem}
  The operad $\Gra_1$ acts on the space of functions on any symplectic vector space. 
 Under this action, $F(m_2)$ above acts as the Moyal product.
 \end{rem}
 
 \begin{lemma}\label{lem:AssGraInjective}
  The map $F$ is indeed a map of operads. It is furthermore injective.
 \end{lemma}
 \begin{proof}
  The fact that the map is a map of operads is analogous to the fact that the Moyal bracket is associative.
 
 To see that $F$ is injective, consider the composition of maps of $\bbS$-modules
 \[
  G: \Poiss :=\Com \circ \Lie \stackrel{\cong}{\to} \Ass \to \Gra_1.
 \]
 On $\Gra$ there is a decreasing complete filtration by the number of edges.
 On $\Com \circ \Lie$ there is a filtration stemming from the arity of $\Com$.
 Clearly it is sufficient to show that the associated graded map $\gr G:\gr \Poiss \to \gr \Gra_1$ is injective.
 But it is not hard to check that $\gr G$ agrees with the operad map $e_1\to \Gra_1$ defined by letting
 \begin{align*}
  \cdot \wedge \cdot &\mapsto 
 \begin{tikzpicture}[baseline=-.65ex]
  \node[ext] (v) at (0,0) {};
  \node[ext] (w) at (.5,0) {};
 \end{tikzpicture}
 &
  [\cdot, \cdot] &\mapsto 
 \begin{tikzpicture}[baseline=-.65ex]
  \node[ext] (v) at (0,0) {};
  \node[ext] (w) at (.5,0) {};
 \draw (v) -- (w);
 \end{tikzpicture}
 \end{align*}
  This latter map is well known to be injective.
 \end{proof}
 
 By precomposing the map $F$ by the inclusion $\Lie\to \Ass$, we obtain a map $\Lie\to \Gra_1$.
 It is described by the Maurer-Cartan element $m$ in the deformation complex 
 \[
  \Def(\Lie\stackrel{0}{\to} \Gra_1).
 \]
 In fact, from this observation it follows once again that the element \eqref{equ:thetaMC} is Maurer-Cartan.
 We can use the machinery of operadic twisting \cite{vastwisting} to produce from the morphism $\Lie\to \Gra_1$
 an operad $\Tw\Gra_1$, together with a map $\Lie\to \Tw\Gra_1$, and with a projection $\Tw\Gra_1\to \Gra_1$.
 Elements of $\Tw\Gra_1$ are series in graphs with two kinds of vertices, internal and external, see Figure \ref{fig:TwGra} for an example.
 
 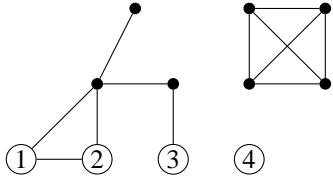
\begin{figure}
 \begin{align*}
 &
  \begin{tikzpicture}
   \node[ext] (v1) at (0,0) {1};
   \node[ext] (v2) at (1,0) {2};
   \node[ext] (v3) at (2,0) {3};
   \node[ext] (v4) at (3,0) {4};
   \node[int] (w1) at (1,1) {};
   \node[int] (w2) at (2,1) {};
   \node[int] (w3) at (1.5,2) {};
   \node[int] (w4) at (3,1) {};
   \node[int] (w5) at (4,1) {};
   \node[int] (w6) at (3,2) {};
   \node[int] (w7) at (4,2) {};
   \draw (v1) edge (v2) edge (w1)
         (w1) edge (w2) edge (w3) edge (v2) 
         (w2) edge (v3)
         (w4) edge (w5) edge (w6) edge (w7)
         (w5) edge (w6) edge (w7)
         (w6) edge (w7);
  \end{tikzpicture}
 \end{align*}
  \caption{\label{fig:TwGra} An element of $\Tw\Gra$.}
 \end{figure}
 
 We may consider a suboperad $\Graphs_1^\Theta$ that consists of series of graphs such that in each connected component there is at least one external vertex.
 \begin{prop}\label{prop:AssGraqiso}
  The map $\Ass\to\Gra_1$ factors through $\Graphs_1^\Theta$, and the map 
 \[
  \Ass \to \Graphs_1^\Theta
 \]
 is a quasi-isomorphism.
 \end{prop}
 \begin{proof}
  The first fact one shows by direct computation. It is the graphical version of the statement that in an associative algebra the commutator with any element is a derivation.
 
 To see that $\Ass \to \Graphs_1^\Theta$ is a quasi-isomorphism note first that the induced map on cohomology is injective. This is because the map is injective on chains by Lemma \ref{lem:AssGraInjective}, and none of the non-zero elements in the image may be exact because they are composed of graphs without internal vertices.
 Hence it suffices to check that the map is surjective on cohomology. 
 It clearly suffices to check that $\dim(H(\Graphs_1^\Theta(N))\leq N!$
 Again consider the descending complete filtration on $\Graphs_1^\Theta$ by the number of edges. The first term of this filtration is $\Graphs_1$, with $H^1(\Graphs_1)\cong \Poiss$. On later pages of the spectral sequence the cohomology can only get smaller. Hence we may conclude that $\dim H^1(\Graphs_1)\leq \dim e_1(N)=N!$.
 \end{proof}
 
 Note that the dg Lie algebra $\fGCc_1^\Theta$ may be extended by one dimension to a dg Lie algebra 
 \[
 \mathfrak{g} := \mathbb{K} L \oplus \fGCc_1^\Theta
 \]
 where the extra element $L$ is the generator of the grading by loop order, i.e., for a graph $\Gamma$ with $l$ loops we define $[L,\Gamma]=l\cdot \Gamma$.
 Note that the grading by loop order is compatible with the Lie brackets, but not with the differential. As a consequence, the element $L$ is not closed, the image being the $\Theta$-class.
 The action of $\fGCc_1^\Theta$ on $\Graphs_1^\Theta$ may naturally be extended to $\alg g$.
 
The operad $\Graphs_1^\Theta$ furthermore carries a complete descending filtration by the number of edges in graphs. This filtration is equivalent to the Hodge filtration on $\Ass\subset \Graphs_1^\Theta$.
 
 Now consider the operadic deformation complex 
 \[
  \Def(\Ass_\infty \to \Graphs_1^\Theta)
 =
 \prod_{N\geq 1} \Hom_{S_N}(\Ass^\vee(N),\Graphs_1^\Theta(N))[-1]
 =
 \prod_{N\geq 1} \Graphs_1^\Theta(N)[1-N].
 \]
Via the action of $\alg g$ on the target we obtain a map $\alpha:\alg g\to \Def(\Ass_\infty \to \Graphs_1^\Theta)[1]$.
But the complex $\Def(\Ass_\infty \to \Graphs_1^\Theta)$ on the right-hand side is acyclic by the rigidity of the associative operad, i.e., by Lemma \ref{lem:assrigid} above.
The map $\alpha$ is injective on the level of chain complexes, and one can likely use a variant of the methods of \cite{grt} to show that $\alpha$ induces an injection on homology. 
Modulo this statement we hence obtain an alternative proof of Theorem \ref{thm:2} stating that $\alg g$ is acyclic. (The presence of the theta-class in Theorem \ref{thm:2} reflects the non-presence of the additional element $L$ above.)

\section{Explicit canceling of the loop $L_7$}\label{app:7loop}

As we can conclude from the table \ref{tbl:oddcanceling}, the loop $L_7$ cancels a graph with $b=4$. Let us calculate that graph more explicitly. The following diagram describes how the differential acts on $L_7$.
$$
\begin{tikzpicture}[baseline=-.65ex]
 \node (a1) at (0,0) {$L_7$};
 \node (a2) at (2,0) {$0$};
 \node (b1) at (0,-2) {$\tilde R_7^1$};
 \node (b2) at (2,-2) {$\tilde W_7^1$};
 \node (b3) at (4,-2) {$0$};
 \node (c2) at (2,-4) {$\tilde W_7^2$};
 \node (c3) at (4,-4) {$0$};
 \draw (a1) edge[|->] (a2);
 \draw (a1) edge[|->] (b2);
 \draw (b2) edge[|->] (b3);
 \draw (b1) edge[|->] node[above left] {$-$} (b2);
 \draw (a1) edge[|-left to] (c2);
 \draw (b1) edge[|-left to] (c2);
 \draw (c2) edge[|->] (c3);
\end{tikzpicture}
$$
Horizontal arrows describe the standard differential $\delta$, diagonal arrows going one row down\footnote{In the full table it goes two rows down, but maps that change parity of $b$ are $0$, so we skip here graphs with odd $b$.} describe the map $[\Theta,\cdot]$ \footnote{ Here $\Theta$ is the theta graph with two vertices connected by three edges.}, and the diagonal arrow going two rows down describes the map
$[\begin{tikzpicture}[baseline=-.65ex]
 \node[int] (v) at (-.3,0) {};
 \node[int] (w) at (.3,0) {};
 \draw (v) edge[very thick] node[above] {$\scriptstyle 5$} (w);
\end{tikzpicture}
,\cdot]$. Harpooned arrows towards $\tilde W_7^2$ are summed.

Since $\tilde W_7^1=[\Theta,L_7]$ is closed and the cohomology of $\delta$ is $0$ at this point, there exists $\tilde R_7^1$ such that $\delta \tilde R_7^1=-W_7^2$. $\tilde W_7^2=
[\begin{tikzpicture}[baseline=-.65ex]
 \node[int] (v) at (-.3,0) {};
 \node[int] (w) at (.3,0) {};
 \draw (v) edge[very thick] node[above] {$\scriptstyle 5$} (w);
\end{tikzpicture}
,L_7]
+[\Theta,\tilde R_7^1]$ is also closed. We want to show that it is not exact.

Let us write $\tilde W_7^2=W_7^2+O_7^2$ where $W_7^2$ is the part with all vertices at least 3-valent, and $O_7^2$ is the part with at least one vertex of maximum valence 2. That condition is not changed by differential $\delta$, so $\delta \tilde W_7^2=0$ implies $\delta W_7^2=\delta O_7^2=0$. \cite[Proof of Proposition 3.4]{grt} implies that the $O_7^2$ part is exact, so we can restrict our consideration to $W_7^2$.

$\tilde W_7^2$ lives among graphs with $2e=3v$, so $W_7^2$ can alternatively be characterized as a part where all vertices are exactly 3-valent, i.e.\ there are no vertices with valence 4 or more. The map
$[\begin{tikzpicture}[baseline=-.65ex]
 \node[int] (v) at (-.3,0) {};
 \node[int] (w) at (.3,0) {};
 \draw (v) edge[very thick] node[above] {$\scriptstyle 5$} (w);
\end{tikzpicture}
,\cdot]$
always adds a 4 or more-valent vertex, so diagonal arrow going two rows down can be ignored in calculation of $W_7^2$. Furthermore, the map $[\Theta,\cdot]$ can not reduce the valence of the vertices, so in $\tilde R_7^1$ we can ignore graphs with 4 or more-valent vertex. Let the remaining part be denoted by $R_7^1$. 

The differential $\delta$ can not increase the valence of the vertices, so $\delta R_7^1$ does not have 4 or more-valent vertices, too. Let us call a sequence of edges and 2-valent vertices between two 3-valent vertices a chain, and let number of edges in the chain be its length. Differential increases length of chains with even length by $1$. Therefore, $\delta R_7^1$ has at least one chain with odd length different than 1. Let a space spanned with graphs without 4 or more-valent vertices and with at least one chain of odd length different than 1 be denoted $S$.

Let $W_7^1$ be projection of $\tilde W_7^1$ on $S$, and let $O_7^1:=\tilde W_7^1-W_7^1$ be the remaining part. If we can choose $\Gamma$ with $v=7$, $b=2$, such that $W_7^1:=\delta \Gamma$, then $\delta W_7^1=0=\delta O_7^1$, and $O_7^1$ is also $\delta$ of something, say $\Gamma'$. Since $\delta R_7^1\in S$, nothing from $\Gamma'$ is part of $R_7^1$ and we can chose $R_7^1$ to be $-\Gamma$. Now, we have simplified diagram
$$
\begin{tikzpicture}[baseline=-.65ex]
 \node (a1) at (0,0) {$L_7$};
 \node (a2) at (2,0) {$0$};
 \node (b1) at (0,-2) {$R_7^1$};
 \node (b2) at (2,-2) {$W_7^1$};
 \node (b3) at (4,-2) {$0$};
 \node (c2) at (2,-4) {$W_7^2$};
 \node (c3) at (4,-4) {$0$};
 \draw (a1) edge[|->] (a2);
 \draw (a1) edge[|->] (b2);
 \draw (b2) edge[|->] (b3);
 \draw (b1) edge[|->] node[above] {$-$} (b2);
 \draw (b1) edge[|->] (c2);
 \draw (c2) edge[|->] (c3);
\end{tikzpicture}
$$
where horizontal arrows again describes standard differential $\delta$ and diagonal arrows represent projection of $[\Theta,\cdot]$ to $S$ (upper arrow) and to graphs without 2-valent vertices (bottom arrow).

In the following calculation we have to be careful with the signs. Therefore we make the convention that the edges in the main loop go clockwise, and vertices are named with consecutive numbers starting from 1, increasing also clockwise. If there are odd vertices in the loop, the starting vertex does not matter, but if there are even vertices, moving the starting vertex by one step changes the sign of the graph, and therefore we will mark the starting vertex. Vertices that are added by $[\Theta,\cdot]$ go after vertices from the loop, and should be labelled if there are more than 1, and we make the convention that all edges are heading towards them.

Careful calculation leads to
$$
W_7^1=
\begin{tikzpicture}[baseline=0ex]
 \node[int] (t0) at (-90:1) {};
 \node[int] (t1) at (-360/7-90:1) {};
 \node[int] (t2) at (-2*360/7-90:1) {};
 \node[int] (t3) at (-3*360/7-90:1) {};
 \node[int] (t4) at (-4*360/7-90:1) {};
 \node[int] (t5) at (-5*360/7-90:1) {};
 \node[int] (t6) at (-6*360/7-90:1) {};
 \node[int] (s) at (0,0) {};
 \draw (t0) edge (t1);
 \draw (t1) edge (t2);
 \draw (t2) edge (t3);
 \draw (t3) edge (t4);
 \draw (t4) edge (t5);
 \draw (t5) edge (t6);
 \draw (t6) edge (t0);
 \draw (s) edge (t0);
 \draw (s) edge (t1);
 \draw (s) edge (t2);
\end{tikzpicture}
+
\begin{tikzpicture}[baseline=0ex]
 \node[int] (t0) at (-90:1) {};
 \node[int] (t1) at (-360/7-90:1) {};
 \node[int] (t2) at (-2*360/7-90:1) {};
 \node[int] (t3) at (-3*360/7-90:1) {};
 \node[int] (t4) at (-4*360/7-90:1) {};
 \node[int] (t5) at (-5*360/7-90:1) {};
 \node[int] (t6) at (-6*360/7-90:1) {};
 \node[int] (s) at (0,0) {};
 \draw (t0) edge (t1);
 \draw (t1) edge (t2);
 \draw (t2) edge (t3);
 \draw (t3) edge (t4);
 \draw (t4) edge (t5);
 \draw (t5) edge (t6);
 \draw (t6) edge (t0);
 \draw (s) edge (t0);
 \draw (s) edge (t1);
 \draw (s) edge (t4);
\end{tikzpicture}
+
\begin{tikzpicture}[baseline=0ex]
 \node[int] (t0) at (-90:1) {};
 \node[int] (t1) at (-360/7-90:1) {};
 \node[int] (t2) at (-2*360/7-90:1) {};
 \node[int] (t3) at (-3*360/7-90:1) {};
 \node[int] (t4) at (-4*360/7-90:1) {};
 \node[int] (t5) at (-5*360/7-90:1) {};
 \node[int] (t6) at (-6*360/7-90:1) {};
 \node[int] (s) at (0,0) {};
 \draw (t0) edge (t1);
 \draw (t1) edge (t2);
 \draw (t2) edge (t3);
 \draw (t3) edge (t4);
 \draw (t4) edge (t5);
 \draw (t5) edge (t6);
 \draw (t6) edge (t0);
 \draw (s) edge (t0);
 \draw (s) edge (t2);
 \draw (s) edge (t4);
\end{tikzpicture}
$$
Actually, we should multiply everything with $42$ (starting at 7 vertices and permuting added edges in 6 ways), but it is safe to ignore prefactor, as soon as the ratio of factors of different graphs remains the same.

One can check that the following is by differential mapped to $-W_7^1$.
$$
R_7^1=
-
\begin{tikzpicture}[baseline=0ex]
 \node[int] (t0) at (-90:1) {};
 \node[below] at (t0) {$\scriptstyle *$};
 \node[int] (t1) at (-360/6-90:1) {};
 \node[int] (t2) at (-2*360/6-90:1) {};
 \node[int] (t3) at (-3*360/6-90:1) {};
 \node[int] (t4) at (-4*360/6-90:1) {};
 \node[int] (t5) at (-5*360/6-90:1) {};
 \node[int] (s) at (0,0) {};
 \draw (t0) edge (t1);
 \draw (t1) edge (t2);
 \draw (t2) edge (t3);
 \draw (t3) edge (t4);
 \draw (t4) edge (t5);
 \draw (t5) edge (t0);
 \draw (s) edge (t0);
 \draw (s) edge (t1);
 \draw (s) edge (t2);
\end{tikzpicture}
+
\begin{tikzpicture}[baseline=0ex]
 \node[int] (t0) at (-90:1) {};
 \node[below] at (t0) {$\scriptstyle *$};
 \node[int] (t1) at (-360/6-90:1) {};
 \node[int] (t2) at (-2*360/6-90:1) {};
 \node[int] (t3) at (-3*360/6-90:1) {};
 \node[int] (t4) at (-4*360/6-90:1) {};
 \node[int] (t5) at (-5*360/6-90:1) {};
 \node[int] (s) at (0,0) {};
 \draw (t0) edge (t1);
 \draw (t1) edge (t2);
 \draw (t2) edge (t3);
 \draw (t3) edge (t4);
 \draw (t4) edge (t5);
 \draw (t5) edge (t0);
 \draw (s) edge (t0);
 \draw (s) edge (t1);
 \draw (s) edge (t3);
\end{tikzpicture}
-\frac{1}{3}
\begin{tikzpicture}[baseline=0ex]
 \node[int] (t0) at (-90:1) {};
 \node[below] at (t0) {$\scriptstyle *$};
 \node[int] (t1) at (-360/6-90:1) {};
 \node[int] (t2) at (-2*360/6-90:1) {};
 \node[int] (t3) at (-3*360/6-90:1) {};
 \node[int] (t4) at (-4*360/6-90:1) {};
 \node[int] (t5) at (-5*360/6-90:1) {};
 \node[int] (s) at (0,0) {};
 \draw (t0) edge (t1);
 \draw (t1) edge (t2);
 \draw (t2) edge (t3);
 \draw (t3) edge (t4);
 \draw (t4) edge (t5);
 \draw (t5) edge (t0);
 \draw (s) edge (t0);
 \draw (s) edge (t2);
 \draw (s) edge (t4);
\end{tikzpicture}
$$
So, we get
$$
W_7^2=
-
\begin{tikzpicture}[baseline=0ex]
 \node[int] (t0) at (-90:1) {};
 \node[below] at (t0) {$\scriptstyle *$};
 \node[int] (t1) at (-360/6-90:1) {};
 \node[int] (t2) at (-2*360/6-90:1) {};
 \node[int] (t3) at (-3*360/6-90:1) {};
 \node[int] (t4) at (-4*360/6-90:1) {};
 \node[int] (t5) at (-5*360/6-90:1) {};
 \node[int] (s) at (-360/6-90:.2) {};
 \node[below right] at (s) {$\scriptstyle 1$};
 \node[int] (s2) at (-4*360/6-90:.2) {};
 \node[above left] at (s2) {$\scriptstyle 2$};
 \draw (t0) edge (t1);
 \draw (t1) edge (t2);
 \draw (t2) edge (t3);
 \draw (t3) edge (t4);
 \draw (t4) edge (t5);
 \draw (t5) edge (t0);
 \draw (s) edge (t0);
 \draw (s) edge (t1);
 \draw (s) edge (t2);
 \draw (s2) edge (t3);
 \draw (s2) edge (t4);
 \draw (s2) edge (t5);
\end{tikzpicture}
+
\begin{tikzpicture}[baseline=0ex]
 \node[int] (t0) at (-90:1) {};
 \node[below] at (t0) {$\scriptstyle *$};
 \node[int] (t1) at (-360/6-90:1) {};
 \node[int] (t2) at (-2*360/6-90:1) {};
 \node[int] (t3) at (-3*360/6-90:1) {};
 \node[int] (t4) at (-4*360/6-90:1) {};
 \node[int] (t5) at (-5*360/6-90:1) {};
 \node[int] (s) at (-360/6-90:.2) {};
 \node[below right] at (s) {$\scriptstyle 1$};
 \node[int] (s2) at (-4*360/6-90:.2) {};
 \node[above] at (s2) {$\scriptstyle 2$};
 \draw (t0) edge (t1);
 \draw (t1) edge (t2);
 \draw (t2) edge (t3);
 \draw (t3) edge (t4);
 \draw (t4) edge (t5);
 \draw (t5) edge (t0);
 \draw (s) edge (t0);
 \draw (s) edge (t1);
 \draw (s) edge (t3);
 \draw (s2) edge (t2);
 \draw (s2) edge (t4);
 \draw (s2) edge (t5);
\end{tikzpicture}
-\frac{1}{3}
\begin{tikzpicture}[baseline=0ex]
 \node[int] (t0) at (-90:1) {};
 \node[below] at (t0) {$\scriptstyle *$};
 \node[int] (t1) at (-360/6-90:1) {};
 \node[int] (t2) at (-2*360/6-90:1) {};
 \node[int] (t3) at (-3*360/6-90:1) {};
 \node[int] (t4) at (-4*360/6-90:1) {};
 \node[int] (t5) at (-5*360/6-90:1) {};
 \node[int] (s) at (-2*360/6-90:.2) {};
 \node[above] at (s) {$\scriptstyle 1$};
 \node[int] (s2) at (-5*360/6-90:.2) {};
 \node[below] at (s2) {$\scriptstyle 2$};
 \draw (t0) edge (t1);
 \draw (t1) edge (t2);
 \draw (t2) edge (t3);
 \draw (t3) edge (t4);
 \draw (t4) edge (t5);
 \draw (t5) edge (t0);
 \draw (s) edge (t0);
 \draw (s) edge (t2);
 \draw (s) edge (t4);
 \draw (s2) edge (t1);
 \draw (s2) edge (t3);
 \draw (s2) edge (t5);
\end{tikzpicture}
$$

3-valent graphs (all vertices are 3-valent) in $\fGCc_1$ (with odd parity of vertices and edge directions and even parity on edges) can alternatively be described with even parity on all vertices, edges and edge directions, but with cyclic order of edges at each vertex, where reversing any order moves the graph to its negative. We rewrite $W_7^2$ in this description assuming that cyclic order on every vertex is clockwise.
$$
W_7^2=
-
\begin{tikzpicture}[baseline=0ex]
 \node[int] (t0) at (-90:1) {};
 \node[int] (t1) at (-360/6-90:1) {};
 \node[int] (t2) at (-2*360/6-90:1) {};
 \node[int] (t3) at (-3*360/6-90:1) {};
 \node[int] (t4) at (-4*360/6-90:1) {};
 \node[int] (t5) at (-5*360/6-90:1) {};
 \node[int] (s) at (-360/6-90:.2) {};
 \node[int] (s2) at (-4*360/6-90:.2) {};
 \draw (t0) edge (t1);
 \draw (t1) edge (t2);
 \draw (t2) edge (t3);
 \draw (t3) edge (t4);
 \draw (t4) edge (t5);
 \draw (t5) edge (t0);
 \draw (s) edge (t0);
 \draw (s) edge (t1);
 \draw (s) edge (t2);
 \draw (s2) edge (t3);
 \draw (s2) edge (t4);
 \draw (s2) edge (t5);
\end{tikzpicture}
-
\begin{tikzpicture}[baseline=0ex]
 \node[int] (t0) at (-90:1) {};
 \node[int] (t1) at (-360/6-90:1) {};
 \node[int] (t2) at (-2*360/6-90:1) {};
 \node[int] (t3) at (-3*360/6-90:1) {};
 \node[int] (t4) at (-4*360/6-90:1) {};
 \node[int] (t5) at (-5*360/6-90:1) {};
 \node[int] (s) at (-360/6-90:.2) {};
 \node[int] (s2) at (-4*360/6-90:.2) {};
 \draw (t0) edge (t1);
 \draw (t1) edge (t2);
 \draw (t2) edge (t3);
 \draw (t3) edge (t4);
 \draw (t4) edge (t5);
 \draw (t5) edge (t0);
 \draw (s) edge (t0);
 \draw (s) edge (t1);
 \draw (s) edge (t3);
 \draw (s2) edge (t2);
 \draw (s2) edge (t4);
 \draw (s2) edge (t5);
\end{tikzpicture}
+\frac{1}{3}
\begin{tikzpicture}[baseline=0ex]
 \node[int] (t0) at (-90:1) {};
 \node[int] (t1) at (-360/6-90:1) {};
 \node[int] (t2) at (-2*360/6-90:1) {};
 \node[int] (t3) at (-3*360/6-90:1) {};
 \node[int] (t4) at (-4*360/6-90:1) {};
 \node[int] (t5) at (-5*360/6-90:1) {};
 \node[int] (s) at (-2*360/6-90:.2) {};
 \node[int] (s2) at (-5*360/6-90:.2) {};
 \draw (t0) edge (t1);
 \draw (t1) edge (t2);
 \draw (t2) edge (t3);
 \draw (t3) edge (t4);
 \draw (t4) edge (t5);
 \draw (t5) edge (t0);
 \draw (s) edge (t0);
 \draw (s) edge (t2);
 \draw (s) edge (t4);
 \draw (s2) edge (t1);
 \draw (s2) edge (t3);
 \draw (s2) edge (t5);
\end{tikzpicture}
$$

We say the coloring of a 3-valent graph is a map from the set of edges to set of 3 colors (say $\{solid,dotted,waved\}$) such that at each vertex 3 adjacent edges have 3 different colors. We fix a cycle order on colors. Let $\Gamma\in\fGCc_1$ be a 3-valent graph and $C$ its coloring. We define $f(C)=\prod_{v\in V(\Gamma)}f(v)$ where $v$ goes through all vertices of $\Gamma$ and $f(v)$ is $1$ if cycle order of edges at that vertex agrees with cycle order of its colors. We also define
$$
f(\Gamma)=\sum_{C\text{ coloring of }\Gamma}f(C),
$$
and extend it linearly to its linear combinations. The following proposition is easy to check.

\begin{prop}
For every graph $\Gamma\in\fGCc_1$ with exactly one 4-valent vertex and all other 3-valent vertices it holds that $f(\delta\Gamma)=0$.
\end{prop}

We calculate $f(W_7^2)$ by considering all colorings of graphs in linear combination of $f(W_7^2)$. Permuted colorings are also possible, but they head to the same $f$, so it is enough to multiply the final result by $6$.
\begin{multline*}
f(W_7^2)=
-f\left(
\begin{tikzpicture}[baseline=0ex]
 \node[int] (t0) at (-90:1) {};
 \node[int] (t1) at (-360/6-90:1) {};
 \node[int] (t2) at (-2*360/6-90:1) {};
 \node[int] (t3) at (-3*360/6-90:1) {};
 \node[int] (t4) at (-4*360/6-90:1) {};
 \node[int] (t5) at (-5*360/6-90:1) {};
 \node[int] (s) at (-360/6-90:.2) {};
 \node[int] (s2) at (-4*360/6-90:.2) {};
 \draw (t0) edge (t1);
 \draw (t1) edge (t2);
 \draw (t2) edge (t3);
 \draw (t3) edge (t4);
 \draw (t4) edge (t5);
 \draw (t5) edge (t0);
 \draw (s) edge (t0);
 \draw (s) edge (t1);
 \draw (s) edge (t2);
 \draw (s2) edge (t3);
 \draw (s2) edge (t4);
 \draw (s2) edge (t5);
\end{tikzpicture}
\right)-f\left(
\begin{tikzpicture}[baseline=0ex]
 \node[int] (t0) at (-90:1) {};
 \node[int] (t1) at (-360/6-90:1) {};
 \node[int] (t2) at (-2*360/6-90:1) {};
 \node[int] (t3) at (-3*360/6-90:1) {};
 \node[int] (t4) at (-4*360/6-90:1) {};
 \node[int] (t5) at (-5*360/6-90:1) {};
 \node[int] (s) at (-360/6-90:.2) {};
 \node[int] (s2) at (-4*360/6-90:.2) {};
 \draw (t0) edge (t1);
 \draw (t1) edge (t2);
 \draw (t2) edge (t3);
 \draw (t3) edge (t4);
 \draw (t4) edge (t5);
 \draw (t5) edge (t0);
 \draw (s) edge (t0);
 \draw (s) edge (t1);
 \draw (s) edge (t3);
 \draw (s2) edge (t2);
 \draw (s2) edge (t4);
 \draw (s2) edge (t5);
\end{tikzpicture}
\right)+\frac{1}{3}f\left(
\begin{tikzpicture}[baseline=0ex]
 \node[int] (t0) at (-90:1) {};
 \node[int] (t1) at (-360/6-90:1) {};
 \node[int] (t2) at (-2*360/6-90:1) {};
 \node[int] (t3) at (-3*360/6-90:1) {};
 \node[int] (t4) at (-4*360/6-90:1) {};
 \node[int] (t5) at (-5*360/6-90:1) {};
 \node[int] (s) at (-2*360/6-90:.2) {};
 \node[int] (s2) at (-5*360/6-90:.2) {};
 \draw (t0) edge (t1);
 \draw (t1) edge (t2);
 \draw (t2) edge (t3);
 \draw (t3) edge (t4);
 \draw (t4) edge (t5);
 \draw (t5) edge (t0);
 \draw (s) edge (t0);
 \draw (s) edge (t2);
 \draw (s) edge (t4);
 \draw (s2) edge (t1);
 \draw (s2) edge (t3);
 \draw (s2) edge (t5);
\end{tikzpicture}\right)=\\
=-6f\left(
\begin{tikzpicture}[baseline=0ex]
 \node[int] (t0) at (-90:1) {};
 \node[int] (t1) at (-360/6-90:1) {};
 \node[int] (t2) at (-2*360/6-90:1) {};
 \node[int] (t3) at (-3*360/6-90:1) {};
 \node[int] (t4) at (-4*360/6-90:1) {};
 \node[int] (t5) at (-5*360/6-90:1) {};
 \node[int] (s) at (-360/6-90:.2) {};
 \node[int] (s2) at (-4*360/6-90:.2) {};
 \draw (t0) edge (t1);
 \draw (t1) edge[dotted] (t2);
 \draw (t2) edge[snakeit] (t3);
 \draw (t3) edge (t4);
 \draw (t4) edge[dotted] (t5);
 \draw (t5) edge[snakeit] (t0);
 \draw (s) edge[dotted] (t0);
 \draw (s) edge[snakeit] (t1);
 \draw (s) edge (t2);
 \draw (s2) edge[dotted] (t3);
 \draw (s2) edge[snakeit] (t4);
 \draw (s2) edge (t5);
\end{tikzpicture}
\right)-6f\left(
\begin{tikzpicture}[baseline=0ex]
 \node[int] (t0) at (-90:1) {};
 \node[int] (t1) at (-360/6-90:1) {};
 \node[int] (t2) at (-2*360/6-90:1) {};
 \node[int] (t3) at (-3*360/6-90:1) {};
 \node[int] (t4) at (-4*360/6-90:1) {};
 \node[int] (t5) at (-5*360/6-90:1) {};
 \node[int] (s) at (-360/6-90:.2) {};
 \node[int] (s2) at (-4*360/6-90:.2) {};
 \draw (t0) edge (t1);
 \draw (t1) edge[dotted] (t2);
 \draw (t2) edge[snakeit] (t3);
 \draw (t3) edge[dotted] (t4);
 \draw (t4) edge (t5);
 \draw (t5) edge[snakeit] (t0);
 \draw (s) edge[dotted] (t0);
 \draw (s) edge[snakeit] (t1);
 \draw (s) edge (t2);
 \draw (s2) edge (t3);
 \draw (s2) edge[snakeit] (t4);
 \draw (s2) edge[dotted] (t5);
\end{tikzpicture}
\right)-6f\left(
\begin{tikzpicture}[baseline=0ex]
 \node[int] (t0) at (-90:1) {};
 \node[int] (t1) at (-360/6-90:1) {};
 \node[int] (t2) at (-2*360/6-90:1) {};
 \node[int] (t3) at (-3*360/6-90:1) {};
 \node[int] (t4) at (-4*360/6-90:1) {};
 \node[int] (t5) at (-5*360/6-90:1) {};
 \node[int] (s) at (-360/6-90:.2) {};
 \node[int] (s2) at (-4*360/6-90:.2) {};
 \draw (t0) edge (t1);
 \draw (t1) edge[dotted] (t2);
 \draw (t2) edge[snakeit] (t3);
 \draw (t3) edge[dotted] (t4);
 \draw (t4) edge (t5);
 \draw (t5) edge[snakeit] (t0);
 \draw (s) edge[dotted] (t0);
 \draw (s) edge[snakeit] (t1);
 \draw (s) edge (t3);
 \draw (s2) edge (t2);
 \draw (s2) edge[snakeit] (t4);
 \draw (s2) edge[dotted] (t5);
\end{tikzpicture}
\right)+\\
+2f\left(
\begin{tikzpicture}[baseline=0ex]
 \node[int] (t0) at (-90:1) {};
 \node[int] (t1) at (-360/6-90:1) {};
 \node[int] (t2) at (-2*360/6-90:1) {};
 \node[int] (t3) at (-3*360/6-90:1) {};
 \node[int] (t4) at (-4*360/6-90:1) {};
 \node[int] (t5) at (-5*360/6-90:1) {};
 \node[int] (s) at (-2*360/6-90:.2) {};
 \node[int] (s2) at (-5*360/6-90:.2) {};
 \draw (t0) edge (t1);
 \draw (t1) edge[dotted] (t2);
 \draw (t2) edge[snakeit] (t3);
 \draw (t3) edge (t4);
 \draw (t4) edge[dotted] (t5);
 \draw (t5) edge[snakeit] (t0);
 \draw (s) edge[dotted] (t0);
 \draw (s) edge (t2);
 \draw (s) edge[snakeit] (t4);
 \draw (s2) edge[snakeit] (t1);
 \draw (s2) edge[dotted] (t3);
 \draw (s2) edge (t5);
\end{tikzpicture}
\right)+2f\left(
\begin{tikzpicture}[baseline=0ex]
 \node[int] (t0) at (-90:1) {};
 \node[int] (t1) at (-360/6-90:1) {};
 \node[int] (t2) at (-2*360/6-90:1) {};
 \node[int] (t3) at (-3*360/6-90:1) {};
 \node[int] (t4) at (-4*360/6-90:1) {};
 \node[int] (t5) at (-5*360/6-90:1) {};
 \node[int] (s) at (-2*360/6-90:.2) {};
 \node[int] (s2) at (-5*360/6-90:.2) {};
 \draw (t0) edge (t1);
 \draw (t1) edge[dotted] (t2);
 \draw (t2) edge[snakeit] (t3);
 \draw (t3) edge[dotted] (t4);
 \draw (t4) edge (t5);
 \draw (t5) edge[snakeit] (t0);
 \draw (s) edge[dotted] (t0);
 \draw (s) edge (t2);
 \draw (s) edge[snakeit] (t4);
 \draw (s2) edge[snakeit] (t1);
 \draw (s2) edge (t3);
 \draw (s2) edge[dotted] (t5);
\end{tikzpicture}
\right)+2f\left(
\begin{tikzpicture}[baseline=0ex]
 \node[int] (t0) at (-90:1) {};
 \node[int] (t1) at (-360/6-90:1) {};
 \node[int] (t2) at (-2*360/6-90:1) {};
 \node[int] (t3) at (-3*360/6-90:1) {};
 \node[int] (t4) at (-4*360/6-90:1) {};
 \node[int] (t5) at (-5*360/6-90:1) {};
 \node[int] (s) at (-2*360/6-90:.2) {};
 \node[int] (s2) at (-5*360/6-90:.2) {};
 \draw (t0) edge (t1);
 \draw (t1) edge[dotted] (t2);
 \draw (t2) edge (t3);
 \draw (t3) edge[snakeit] (t4);
 \draw (t4) edge[dotted] (t5);
 \draw (t5) edge[snakeit] (t0);
 \draw (s) edge[dotted] (t0);
 \draw (s) edge[snakeit] (t2);
 \draw (s) edge (t4);
 \draw (s2) edge[snakeit] (t1);
 \draw (s2) edge[dotted] (t3);
 \draw (s2) edge (t5);
\end{tikzpicture}
\right)+2f\left(
\begin{tikzpicture}[baseline=0ex]
 \node[int] (t0) at (-90:1) {};
 \node[int] (t1) at (-360/6-90:1) {};
 \node[int] (t2) at (-2*360/6-90:1) {};
 \node[int] (t3) at (-3*360/6-90:1) {};
 \node[int] (t4) at (-4*360/6-90:1) {};
 \node[int] (t5) at (-5*360/6-90:1) {};
 \node[int] (s) at (-2*360/6-90:.2) {};
 \node[int] (s2) at (-5*360/6-90:.2) {};
 \draw (t0) edge (t1);
 \draw (t1) edge[snakeit] (t2);
 \draw (t2) edge[dotted] (t3);
 \draw (t3) edge (t4);
 \draw (t4) edge[dotted] (t5);
 \draw (t5) edge[snakeit] (t0);
 \draw (s) edge[dotted] (t0);
 \draw (s) edge (t2);
 \draw (s) edge[snakeit] (t4);
 \draw (s2) edge[dotted] (t1);
 \draw (s2) edge[snakeit] (t3);
 \draw (s2) edge (t5);
\end{tikzpicture}
\right)
=-6-6-6+2+2+2+2=-10
\end{multline*}

Therefore, the proposition implies that $W_7^2$ can not be exact, what needed to be demonstrated.

\section{A note on the convergence of spectral sequences}\label{app:specseq}

In this appendix we discuss convergence of spectral sequences needed in this paper.

Let $\K$ be a field and $C$ a chain complex of vector spaces over $\K$:
\begin{equation*}
\begin{tikzpicture}
\matrix[diagram](m){\dots & C_{-1} & C_0 & C_1 & C_2 & \dots \\};
\draw [->] (m-1-1) edge node[above]{$d_{-2}$} (m-1-2);
\draw [->] (m-1-2) edge node[above]{$d_{-1}$} (m-1-3);
\draw [->] (m-1-3) edge node[above]{$d_0$} (m-1-4);
\draw [->] (m-1-4) edge node[above]{$d_1$} (m-1-5);
\draw [->] (m-1-5) edge node[above]{$d_2$} (m-1-6);
\end{tikzpicture}
\end{equation*}
Suppose that chain elements can be expressed as $C_q=\prod_{p=0}^\infty C_{p,q}$ and that differential $d_q\colon C_q\rightarrow C_{q+1}$ can be expressed as a direct product of $d_{p,q}^r\colon C_{p,q}\rightarrow C_{p+r,q+1}$ for $r\geq 0$, as drawn in the diagram:
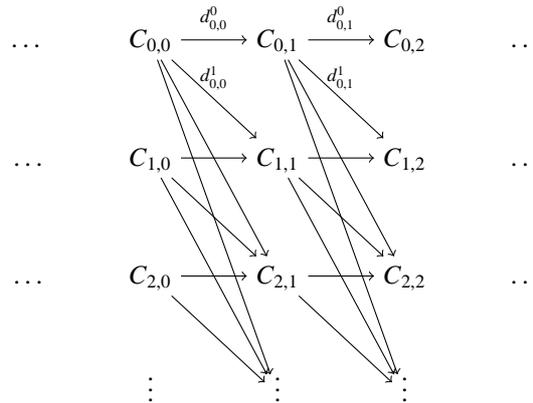
\begin{equation*}
\begin{tikzpicture}
\matrix[diagram](m){\dots & C_{0,0} & C_{0,1} & C_{0,2} & \dots \\
\dots & C_{1,0} & C_{1,1} & C_{1,2} & \dots \\
\dots & C_{2,0} & C_{2,1} & C_{2,2} & \dots \\
 & \vdots & \vdots & \vdots & \\};
\draw [->] (m-1-2) edge node[above]{$\scriptstyle d_{0,0}^0$} (m-1-3);
\draw [->] (m-1-2) edge node[above]{$\scriptstyle d_{0,0}^1$} (m-2-3);
\draw [->] (m-1-2) edge (m-3-3);
\draw [->] (m-1-2) edge (m-4-3);
\draw [->] (m-1-3) edge node[above]{$\scriptstyle d_{0,1}^0$} (m-1-4);
\draw [->] (m-1-3) edge node[above]{$\scriptstyle d_{0,1}^1$} (m-2-4);
\draw [->] (m-1-3) edge (m-3-4);
\draw [->] (m-1-3) edge (m-4-4);
\draw [->] (m-2-2) edge (m-2-3);
\draw [->] (m-2-2) edge (m-3-3);
\draw [->] (m-2-2) edge (m-4-3);
\draw [->] (m-2-3) edge (m-2-4);
\draw [->] (m-2-3) edge (m-3-4);
\draw [->] (m-2-3) edge (m-4-4);
\draw [->] (m-3-2) edge (m-3-3);
\draw [->] (m-3-2) edge (m-4-3);
\draw [->] (m-3-3) edge (m-3-4);
\draw [->] (m-3-3) edge (m-4-4);
\end{tikzpicture}
\end{equation*}
Note that this defines complete filtration bounded above $F_nC$ of $C$ where $F_nC_q=\prod_{p=n}^\infty C_{p,q}$. One can set up a spectral sequence to this filtration. All spectral sequences in this paper are of this kind. We need two propositions to ensure the correct convergence of them.

\begin{prop}\label{prop:App1}
If for every $q\in\Z$ there are only finitely many $q\in\N$ such that $C_{p,q}\neq 0$, spectral sequence converges to the homology of the complex $C$.
\end{prop}
\begin{proof}
Under condition from the proposition, the filtration is bounded, and \cite[Clasical Convergence Theorem 5.51]{Weibel} implies the result.
\end{proof}

\begin{prop}\label{prop:App2}
If $C_{p,q}$ are finitely dimensional for all $p$ and $q$, spectral sequence weakly converges to the homology of the complex $C$.
\end{prop}
\begin{proof}
One can see that the finite dimensionality of $C_{p,q}$ implies regularity of the spectral sequence (\cite[Definition 5.2.10]{Weibel}). Then \cite[Complete Convergence Theorem 5.5.10]{Weibel} implies the result.
\end{proof}

\bibliographystyle{plain}

\end{document}